\documentclass{amsart}
\usepackage{mathrsfs}
\usepackage{amscd, amssymb, amsmath, amsthm, amsfonts}
\usepackage{latexsym, graphics, graphicx, psfrag}
\usepackage{color,xcolor,colortbl}
\usepackage[all]{xy}

\newtheorem{thm}{Theorem}[section]
\newtheorem{col}[thm]{Corollary}
\newtheorem{lem}[thm]{Lemma}
\newtheorem{prop}[thm]{Proposition}
\newtheorem{que}[thm]{Question}
\theoremstyle{definition}
\newtheorem{defn}[thm]{Definition}
\newtheorem{construction}[thm]{Construction}
\newtheorem{parameter}[thm]{Parameter}
\theoremstyle{remark}
\newtheorem{rem}[thm]{Remark}

\usepackage[linkcolor=red, citecolor=blue]{hyperref}

\title[Virtual Domination of $3$-manifolds]{Virtual Domination of $3$-manifolds}

\author[Hongbin~Sun]{Hongbin Sun}
\address{%
    Mathematics Department\\
    Princeton University\\
    Princeton, NJ 08544, USA}
\email{%
    hongbins@math.princeton.edu}

\date{}

\begin{document}

\begin{abstract} For any closed oriented hyperbolic $3$-manifold $M$, and any closed oriented $3$-manifold $N$, we will show that $M$ admits a finite cover $M'$, such that there exists a degree-$2$ map $f:M'\rightarrow N$, i.e. $M$ virtually $2$-dominates $N$.
\end{abstract}

\maketitle


\section{Introduction}
\subsection{Background and the Main Result}

Traditionally, essential codimension-$1$ objects in $3$-manifolds, e.g. incompressible surfaces, taut foliations, essential laminations, are very important and interesting objects in $3$-manifold topology. There are various methods to construct such essential codimension-$1$ objects, and most of the constructions use topological methods.

Recently, in \cite{KM1}, Kahn and Markovic used hyperbolic geometry and dynamical system to show the following Surface Subgroup Theorem: for any closed hyperbolic $3$-manifold $M$, there exists a closed hyperbolic surface $S$, such that there is a $\pi_1$-injective almost totally geodesic immersion $ S\looparrowright M$. Building on Wise's work (\cite{Wi}), Agol showed that the groups of hyperbolic $3$-manifolds are virtually special and LERF (\cite{Ag}). Agol's result admits us to find a finite cover of $M$, such that $S$ lifts to an embedded incompressible surface, which solves Thurston's Virtual Haken Conjecture (\cite{Th2}). Actually, it is the Surface Subgroup Theorem that admits Wise's machine on geometric group theory available for studying closed hyperbolic $3$-manifolds.

In \cite{Su}, for any closed hyperbolic $3$-manifold $M$, we used Kahn-Markovic surfaces to construct an immersed $\pi_1$-injective $2$-complex $X_n\looparrowright M$. Here the $2$-complex $X_n$ is a local model of homological $\mathbb{Z}_n$-torsion. Then the results of Agol (\cite{Ag}) and Haglund-Wise (\cite{HW}) can be applied to $X_n\looparrowright M$, and we showed the following result: for any finite abelian group $A$, and any closed hyperbolic $3$-manifold $M$, $M$ admits a finite cover $M'$, such that $A$ is a direct summand of $Tor(H_1(M';\mathbb{Z}))$ (see \cite{Su}).

The proof of the above result suggests us to construct some other type of immersed $\pi_1$-injective $2$-complexes in closed hyperbolic $3$-manifolds. Then LERF or other virtual properties of hyperbolic $3$-manifolds will imply some other nice results. In this paper, we will give another application of this idea, and show the following result.

\begin{thm}\label{main}
For any closed oriented hyperbolic $3$-manifold $M$, and any closed oriented $3$-manifold $N$, $M$ admits a finite cover $M'$, such that there exists a degree-$2$ map $f:M'\rightarrow N$, i.e. $M$ virtually $2$-dominates $N$.
\end{thm}

\begin{rem}
In a previous version of this paper, we used results in \cite{Ga}, and could only show that any closed oriented hyperbolic $3$-manifold virtually dominates any closed oriented $3$-manifold, but with no bound on the degree of the non-zero degree map. The author wants to thank Ian Agol for introducing him the results in \cite{HLMW}, which admits us to get the virtual $2$-domination result.
\end{rem}

Theorem \ref{main} answers a question asked by Agol: whether any closed hyperbolic $3$-manifold virtually dominates any closed $3$-manifold, which was a possible approach to prove the Virtual Haken Conjecture. In some sense, the existence of a nonzero degree map from one $3$-manifold $M$ to another $3$-manifold $N$ implies that $M$ is (topologically) more complicated than $N$. So Theorem \ref{main} implies that any closed hyperbolic $3$-manifold is virtually more complicated than any closed $3$-manifold. For more information about history, results and questions on non-zero degree maps between $3$-manifolds, see the survey paper \cite{Wa}.

Since closed hyperbolic $3$-manifolds have virtually positive first betti-number (\cite{Ag}), we can suppose that $M$ has already satisfied $b_1(M)>0$, then the following immediate corollary holds.

\begin{col}\label{even}
For any even number $2d$, any closed oriented hyperbolic $3$-manifold $M$, and any closed oriented $3$-manifold $N$, $M$ admits a finite cover $M''$, such that there exists a degree-$2d$ map $g:M''\rightarrow N$, i.e. $M$ virtually $2d$-dominates $N$.
\end{col}

Theorem \ref{main} also answers Question 8.2 in \cite{DLW} for closed hyperbolic $3$-manifolds, by taking $N$ to be any closed $3$-manifold which support the $\widetilde{PSL_2(\mathbb{R})}$ geometry.

\begin{col}\label{volume}
For any closed hyperbolic $3$-manifold $M$, $M$ admits a finite cover $M'''$, such that $M'''$ has positive Seifert volume ($Iso_e\widetilde{SL_2(\mathbb{R})}$-representation volume).
\end{col}

\begin{rem}
If we take $N$ to be connected sum of lens spaces, then Theorem \ref{main} implies a weaker version of the main result in \cite{Su}: for any closed hyperbolic $3$-manifold $M$, and any finite abelian group, $M$ admits a finite cover $M'$, such that $A$ is embedded into $Tor(H_1(N;\mathbb{Z}))$.
\end{rem}

Since closed $3$-manifolds with vanishing simplicial volume do not virtually dominate any closed hyperbolic $3$-manifold (by considering the simplicial volume), it is natural to ask the following question.

\begin{que}\label{sim}
For any closed oriented $3$-manifold $M$ with positive simplicial volume, whether $M$ virtually dominates any closed oriented $3$-manifold?
\end{que}

To give a positive answer to this question, it suffices to show that any irreducible closed oriented $3$-manifold with a hyperbolic piece in its JSJ decomposition virtually dominates some closed oriented hyperbolic $3$-manifold. Theorem 1.6 (1) in \cite{DLW} gives some evidence for Question \ref{sim}. If one can extend Kahn-Markovic and Liu-Markovic's theories to cusped hyperbolic $3$-manifolds, then Question \ref{sim} can be confirmed.

It is also natural to ask whether the following result holds, which tries to strengthen Theorem \ref{main} quantitatively.

\begin{que}
For any closed oriented hyperbolic $3$-manifold $M$ (or closed oriented $3$-manifold with positive simplicial volume), whether $M$ virtually $1$-dominates any closed oriented $3$-manifold?
\end{que}

We can only prove that virtual $2$-domination exists, but could not promote it to virtual $1$-domination, basically because of the $\mathbb{Z}_2$-valued invariant $\sigma$ (see Corollary \ref{bound}), which was introduced in Theorem 1.4 of \cite{LM}.

In Section \ref{review}, we will give a quick review of the results in \cite{KM1}, \cite{KM2} and \cite{LM} which are necessary for this paper. In Section \ref{construction1}, we will give the topological part of the proof of Theorem \ref{main}. In Section \ref{handle}, for any closed oriented $3$-manifold $N$, we will construct a nice handle structure on it, by using results in \cite{HLMW}. In Section \ref{construction2}, for any closed hyperbolic $3$-manifold $M$, we will construct a $2$-complex $Z$ (hinted by the handle structure of $N$) and a $\pi_1$-injective immersion $j:Z\looparrowright M$. In Section \ref{pinching}, we will describe how does the existence of this immersed $\pi_1$-injective $2$-complex implies Theorem \ref{main}. The proof of the $\pi_1$-injectivity of $j:Z\looparrowright M$ will be delayed to Section \ref{injectivity}.

\subsection{Sketch of the Proof} Here we give a brief sketch of the proof of Theorem \ref{main}.

In \cite{Th1}, Thurston described a hyperbolic $3$-orbifold $M_0$, whose underlying space is $S^3$, and the singular set is the Borromean rings with indices $4$. Moreover, in \cite{HLMW}, it is shown that $M_0$ has the following universal property.

\begin{thm}[\cite{HLMW}]\label{HLMW}
For any closed oriented $3$-manifold $N$, there is a finite index subgroup $\Gamma\subset \pi_1(M_0)\subset PSL_2(\mathbb{C})$, such that $N$ is homeomorphic to $\mathbb{H}^3/\Gamma$ with respect to their orientations. Here we ignore the orbifold structure of $\mathbb{H}^3/\Gamma$, and just think it as a $3$-manifold.
\end{thm}

In Section \ref{handle}, we will construct an orbifold handle structure (see Definition \ref{ohs}) for $M_0$, by following the geometry of the regular dodecahedron (or equivalently, regular icosahedron). Then this orbifold handle structure of $M_0$ is lifted to an orbifold handle structure of $\mathbb{H}^3/\Gamma$, by the finite sheet cover provided by Theorem \ref{HLMW}. Then we have a nice handle structure of $N$, which is related with the geometry of the regular icosahedron.

For a closed $3$-manifold (orbifold) $P$ endowed with an (orbifold) handle structure, we will use $P^{(1)}$ to denote the union of $0$- and $1$-handles, and use $P^{(2)}$ to denote the union of $0$-, $1$- and $2$-handles.

We will also construct a $2$-subcomplex $X\subset M_0$ which is a deformation retract of $M_0^{(2)}$, and $X$ lifts to a $2$-subcomplex $Y\subset N$ which is a deformation retract of $N^{(2)}$.

For any closed oriented hyperbolic $3$-manifold $M$ and any point $p\in M$, choose twelve unit vectors in $T^1_p M$ which correspond with the normal vectors of the twelve faces of the regular dodecahedron. By using the exponential mixing property of the frame flow (\cite{Mo}, \cite{Po}), we can construct an immersion of the $1$-skeleton $X^{(1)}$ of $X$ into $M$, denoted by $j': X^{(1)}\looparrowright M$. This construction of $j': X^{(1)}\looparrowright M$ is hinted by the geometry of the regular dodecahedron, and satisfies the following conditions.
\begin{itemize}
\item The $0$-cell of $X^{(1)}$ is mapped to $p$.
\item The six $1$-cells are mapped to geodesic arcs in $M$ based at $p$, and their tangent vectors at $p$ are very close to two of those twelve unit vectors (corresponding with the $1$-handles of the orbifold handle structure of $M_0$).
\item The image of $1$-cells of $X^{(1)}$ are homologous to $0$ in $H_1(M;\mathbb{Z})$.
\item There exists a large real number $R>0$, such that the following holds. For any $2$-cell of $X$, the image of the boundary of this $2$-cell in $M$ is homotopic to a closed geodesic whose complex length is very close to $4R$ or $R$ (dependsing on whether this $2$-cell intersects with the singular set of $M_0$ or not).
\end{itemize}

Since the $1$-skeleton $Y^{(1)}$ of $Y$ is a finite cover of $X^{(1)}$, $j': X^{(1)}\looparrowright M$ induces an immersion $Y^{(1)}\looparrowright M$. We take two copies of the immersed $2$-complex $Y^{(1)}\looparrowright M$, and denote them by $Y_1^{(1)}\looparrowright M$ and $Y_2^{(1)}\looparrowright M$. For any $2$-cell $c_i$ in $Y$ (with an arbitrary orientation), let $\gamma_i$ be the oriented closed geodesic homotopic to the image of $\partial c_i$ in $M$. Take two copies of $\gamma_i$ and denote them by $\gamma_i^1$ and $\gamma_i^2$, then the recent result in \cite{LM} (see Corollary \ref{bound}) implies that $\gamma_i^1$ and $\gamma_i^2$ bound an immersed oriented almost totally geodesic subsurface $S_i$ in $M$ (possibly disconnected).

By pasting the immersed $1$-complexes $Y_1^{(1)}$ and $Y_2^{(1)}$, almost totally geodesic surfaces $\{S_i\}$, and almost totally geodesic annuli connecting $\gamma_i^j$ with $Y_j^{(1)}$ for $j=1,2$, we get an immersed $2$-complex $j:Z\looparrowright M$. The $2$-complex $Z$ is connected and almost totally geodesic in $M$ except along $Y_1^{(1)} \cup Y_2^{(1)}\subset Z$. Moreover, if the surfaces $S_i$ are complicated enough, then $j_*:\pi_1(Z)\rightarrow \pi_1(M)$ is injective.

Since Agol showed that the groups of hyperbolic $3$-manifolds are LERF (\cite{Ag}), $M$ admits a finite cover $M'$ such that a geometric neighborhood $Z$ is embedded into $M'$, and this neighborhood is denoted by $K$. Let $A_i$ be the annulus on $\partial N^{(1)}$ where the $i$th $2$-handle is attached, then $K$ is homeomorphic to the quotient space of two copies of $N^{(1)}$ and the disjoint union of $\{S_i\times I\}$, by pasting $(\partial S_i)\times I$ to the two copies of $A_i$.

Then there is a proper degree-$2$ map $h:(K,\partial K) \rightarrow (N^{(2)},\partial N^{(2)})$, which maps the two copies of $N^{(1)}$ in $K$ to $N^{(1)}\subset N^{(2)}$ by identity map, and maps $S_i\times I$ to the corresponding $2$-handle of $N$. Then by pinching the components of $M'\setminus K$ to wedge of $3$-balls, $h:K\rightarrow N^{(2)}$ can be extended to a degree-$2$ map $f:M'\rightarrow N$, as desired.

{\bf Acknowledgement:} The author is grateful to his advisor David Gabai for many helpful conversations, suggestions and encouragements, and thanks Shicheng Wang for introducing the author to the field of non-zero degree map between $3$-manifolds. The author thanks Ian Agol and the organizers of the conference "Cube complexes and $3$-manifolds", which is held at the University of Illinois at Chicago, since the author first learned about Agol's question in his talk during this conference. The author also thanks Ian Agol and Shicheng Wang for providing help on math literatures, thanks Yi Liu for pointing out Lemma \ref{nullhomologous} to the author, and thanks Stefan Friedl for comments on a previous draft.

\section{Kahn-Markovic and Liu-Markovic's Works on Constructing Almost Totally Geodesic Subsurfaces}\label{review}

In this section, we give a quick review of Kahn-Markovic and Liu-Markovic's works on constructing almost totally geodesic subsurfaces in closed hyperbolic $3$-manifolds. All the material in this section can be found in \cite{KM1}, \cite{KM2} and \cite{LM}, and we only state the results that are necessary for this paper.

In \cite{KM1}, Kahn and Markovic proved the following Surface Subgroup Theorem, which is the first step to prove Thurston's Virtual Haken and Virtual Fibered Conjectures. (The conjectures were raised in \cite{Th2}, and settled in \cite{Ag}).

\begin{thm}[\cite{KM1}]\label{surface}
For any closed hyperbolic $3$-manifold $M$, there exists an immersed closed hyperbolic surface $f:S\looparrowright M$, such that $f_*:\pi_1(S)\rightarrow \pi_1(M)$ is an injective map.
\end{thm}

Actually, the surfaces constructed in Theorem \ref{surface} are almost totally geodesic subsurfaces, which are constructed by pasting oriented {\it good pants} together in an almost totally geodesic way. Given a closed hyperbolic $3$-manifold $M$, for any small number $\epsilon>0$ and large number $R>0$, the set of $(R,\epsilon)$-good pants ${\bold \Pi}_{R,\epsilon}$ consists of homotopy classes of immersed oriented pair of pants $\Pi \looparrowright M$, such that the three cuffs of $\Pi$ are mapped to closed geodesics $\gamma_i$, and $|\bold{hl}_{\Pi}(\gamma_i)-\frac{R}{2}|<\epsilon$ holds for $i=1,2,3$ (for the definition of $\bold{hl}_{\Pi}(\gamma)$, see \cite{KM1} page 1131).

These oriented good pants are pasted along oriented {\it good curves}, which are oriented closed geodesics in $M$ with complex length $2\epsilon$-close to $R$, and the set consists of such good curves is denoted by ${\bold \Gamma}_{R,\epsilon}$. For two good pants pasting along a good curve, an almost $1$-shift ($|s(C)-1|<\frac{\epsilon}{R}$) should be applied (for the definition of $s(C)$, see \cite{KM1} page 1132). This almost $1$-shift is an essential condition to guarantee that $f_*:\pi_1(S)\rightarrow \pi_1(M)$ is injective. In \cite{KM1}, Kahn and Markovic showed that, for any good curve $\gamma$, the feet of good pants on $\gamma$ are equidistributed on the (half) unit normal bundle of $\gamma$, i.e. the counting measure of the feet of good pants is close to some scaling of the Lebesgue measure on the (half) unit normal bundle. So the immersed surface satisfying $|s(C)-1|<\frac{\epsilon}{R}$ can be constructed. For more precise statement, see Theorem 3.4 of \cite{KM1}.

After proving the Surface Subgroup Theorem, Kahn and Markovic worked on $1$-dimensional lower, and proved the Ehrenpreis conjecture:
\begin{thm}[\cite{KM2}]\label{Ehrenpreis}
Let $S$ and $T$ be two closed Riemann surfaces with negative Euler characteristics. Then for any $k>1$, $S$ and $T$ admit finite covers $S_1$ and $T_1$ respectively, such that there exists a $k$-quasiconformal map $f:S_1 \rightarrow T_1$.
\end{thm}

To prove Theorem \ref{Ehrenpreis}, Kahn and Markovic showed the following theorem.

\begin{thm}[\cite{KM2}]\label{model}
Let $S$ be a closed hyperbolic Riemann surface. Then for any $k>1$, there exists $R_0(K,S)>0$, such that for any $R>R_0(K,S)$, the following statement holds. There is a closed hyperbolic Riemann surface $O$ with pants decomposition $\mathcal{C}$, satisfying $\bold{l}(C)=R$ and $s(C)=1$ for any $C\in \mathcal{C}$, such that there exists a $k$-quasiconformal map $g:O\rightarrow S_1$ from $O$ to some finite cover $S_1$ of $S$.
\end{thm}

As the proof of Theorem \ref{surface}, one can try to paste immersed good pants in $S$ along good curves, to obtain a finite cover of $S$ which satisfies $|\bold{l}(C)-R|<2\epsilon$ and $|s(C)-1|<\frac{\epsilon}{R}$, then this finite cover gives the desired $k$-quasiconformal map. To make the pasting construction works, one need to make sure that, for any good curve $\gamma$ in $S$, the number of good pants to the left of $\gamma$ should exactly equal the number of good pants to the right of $\gamma$. However, the equidistribution result only claims that these two numbers are very close to each other, but may not be equal. In dimension $3$, since the unit normal bundle of a closed geodesic is connected (a topological torus), such a problem does not appear. However, one does need to take care of this imbalance problem in dimension $2$.

To deal with the problem, Kahn and Markovic studied the {\it good pants homology}. Two elements $c_1,c_2$ in $\mathbb{R}{\bold \Gamma}_{R,\epsilon}$ are equivalent in the ${\bold \Pi}_{R,\epsilon}$-good pants homology if there exists $w\in \mathbb{R}{\bold \Pi}_{R,\epsilon}$ such that $\partial w=c_1-c_2$. Then the ${\bold \Pi}_{R,\epsilon}$-good pants homology is defined to be $\mathbb{R}{\bold \Gamma}_{R,\epsilon}/\partial \mathbb{R}{\bold \Pi}_{R,\epsilon}$, and the following result was proved in \cite{KM2}.

\begin{thm}[\cite{KM2}]\label{homology1}
Given a closed hyperbolic surface $S$, for small enough $\epsilon>0$ depending on $S$ and large enough $R>0$ depending on $\epsilon$ and $S$, the ${\bold \Pi}_{R,\epsilon}$-good pants homology is naturally isomorphic with $H_1(S;\mathbb{R})$. Moreover, the same statement holds if the real coefficient is replaced by rational coefficient.
\end{thm}

By using Theorem \ref{homology1} and randomization technique, Kahn and Markovic could take care of the imbalance problem, and proved Theorem \ref{model}.

Later, in \cite{LM}, Liu and Markovic pushed the theory of good pants homology from dimension $2$ to dimension $3$. Instead of working on real (or rational) coefficient as in Theorem \ref{homology1}, they worked on homology with integer coefficient. In \cite{LM}, they proved the following theorem (comparing with Theorem \ref{homology1}).

\begin{thm}[\cite{LM}]\label{homology2}
Given a closed oriented hyperbolic $3$-manifold $M$, for small enough $\epsilon>0$ depending on $M$ and large enough $R>0$ depending on $\epsilon$ and $M$, let $\Omega_{R,\epsilon}(M)$ be the quotient $\mathbb{Z}{\bold \Gamma}_{R,\epsilon}/\partial\mathbb{Z}{\bold \Pi}_{R,\epsilon}$ (which is called the pants cobordism group in \cite{LM}). Then there is a natural isomorphism $\Phi:\Omega_{R,\epsilon}(M)\rightarrow H_1(SO(M);\mathbb{Z})$. Here $SO(M)$ denotes the bundle of orthonormal frames of $M$ which give the orientation of $M$.
\end{thm}

For any $\gamma \in {\bold \Gamma}_{R,\epsilon}$, a map $\hat{\gamma}:S^1\rightarrow SO(M)$ is defined in \cite{LM}, which is called the {\it canonical lifting} of $\gamma$ and gives the definition of $\Phi$. To define $\hat{\gamma}$, choose a point $p\in \gamma$, and take an orthonormal frame $e_p=(\vec{t},\vec{n},\vec{t}\times \vec{n})$ in $T_pM$ such that $\vec{t}$ is tangent to $\gamma$ and follow the orientation of $\gamma$. Then $\hat{\gamma}$ is defined by first flow $e_p$ once around $\gamma$ by parallel transportation, then do a counterclockwise $2\pi$-rotation about $\vec{n}$, finally travel back to $e_p$ along a $2\epsilon$-short path. Then the isomorphism $\Phi:\Omega_{R,\epsilon}(M)\rightarrow H_1(SO(M);\mathbb{Z})$ is defined by $\gamma\rightarrow [\hat{\gamma}]\in H_1(SO(M);\mathbb{Z})$.

The following definition was implicitly given in \cite{LM}, and we restate it here.

\begin{defn}
Let $S$ be a compact oriented hyperbolic surface (closed or with boundary), and $\mathcal{C}=\{C_1,\cdots,C_n\}$ be a family of disjoint simple closed curves on $S$ which give a pants decomposition of $S$, such that each boundary component of $S$ corresponds with some $C_i$. An immersion $f:S\looparrowright M$ to a closed hyperbolic $3$-manifold $M$ is called an immersed {\it $(R,\epsilon)$-almost totally geodesic subsurface} in $M$ if the following conditions hold.
\begin{itemize}
\item For any pair of pants $\Pi$ in $S\setminus \mathcal{C}$ and $C\in \mathcal{C}$ be a boundary component of $\Pi$, $|\bold{hl}_{\Pi}(C)-\frac{R}{2}|<\epsilon$ holds.
\item For any two pair of pants sharing some $C\in \mathcal{C}$ as their boundary component, $|s(C)-1|<\frac{\epsilon}{R}$ holds.
\end{itemize}
\end{defn}

In Lemma 3.8 of \cite{LM}, it is shown that for any small enough $\epsilon>0$ and large enough $R>0$, immersed $(R,\epsilon)$-almost totally geodesic surfaces are $\pi_1$-injective.

As a corollary of Theorem \ref{homology2}, we have the following result (Corollary \ref{bound}). The statement of Corollary \ref{bound} is very similar with Theorem 1.4 of \cite{LM}, and the idea of the proof is scattered in a few proofs in \cite{LM}. We reorganize the material to make the following statement, which is convenient for the application in our case. We also give a brief proof here, by following the idea of \cite{LM}.

\begin{col}\label{bound}
Let $M$ be a closed hyperbolic $3$-manifold, then for any small enough $\epsilon>0$ depending on $M$, and any large enough $R>0$ depending on $\epsilon$ and $M$, the following statement holds. For any null-homologous oriented $(R,\epsilon)$-multicurve $L\in \mathbb{Z}{\bold \Gamma}_{R,\epsilon}$, there is a nontrivial invariant $\sigma(L)\in \mathbb{Z}_2$ defined, such that the following properties hold.
\begin{itemize}
\item $\sigma(L_1\sqcup L_2)=\sigma(L_1)+\sigma(L_2)$.
\item $\sigma(L)$ vanishes if and only if $L$ bounds an immersed compact oriented $(R,\epsilon)$ almost totally geodesic surface $S$ in $M$. Moreover, if we associate each component $l_i$ ($i=1,\cdots,n$) of $L$ with a normal vector $\vec{v}_i\in T^1_{p_i} M$ for some $p_i\in l_i$, then the surface $S$ can be constructed to satisfy the following condition. Let $C_i$ be the boundary component of $S$ which is mapped to $l_i$, and $\Pi_i$ be the pair of pants in $S$ with $C_i$ as one of its cuff, then one of the feet of $\Pi_i$ on $l_i$ is $\frac{\epsilon}{R}$-close to $v_i$.
\end{itemize}
\end{col}

\begin{proof}
Since $L$ is null-homologous, $\Phi(L)\in H_1(SO(3);\mathbb{Z})\subset H_1(SO(M);\mathbb{Z})$. So $\sigma(L)$ is simply defined to be $\Phi(L)\in H_1(SO(3);\mathbb{Z})\cong \mathbb{Z}_2$. All the statements are clearly true by Theorem \ref{homology2}, except that $\sigma(L)=0$ implies that $L$ bounds an immersed compact oriented $(R,\epsilon)$-almost totally geodesic subsurface in $M$, and the moreover part.

Theorem \ref{homology2} implies that there exists $w\in \mathbb{Z}{\bold \Pi}_{R,\epsilon}$, such that $\partial w=L$. So $w$ defines an integer valued measure on ${\bold \Pi}_{R,\epsilon}$. By adding the ubiquitous, $(R,\frac{\epsilon}{2})$-nearly even footed, integer valued measure $\mu_0$ on ${\bold \Pi}_{R,\epsilon}$ in Theorem 2.10 of \cite{LM} (see also the definitions therein), $w+n\mu_0$ is still ubiquitous, $(R,\epsilon)$-nearly even footed for large enough positive integer $n$, and $\partial(w+n\mu_0)=L$. So $w+n\mu_0$ admits an $(R,\epsilon)$-nearly unit shearing gluing (satisfying $|s(C)-1|<\frac{\epsilon}{R}$), which gives an immersed $(R,\epsilon)$-almost totally geodesic surface $S$ in $M$ with $\partial S=L$. By Lemma 3.8 of \cite{LM}, each component of $S$ is $\pi_1$-injective.

For the moreover part, we will construct an immersed surface $S'\looparrowright M$ which contains the surface $S\looparrowright M$ we have constructed as a subsurface. Let $C_i$ be the boundary component of $S$ which is mapped to $l_i$, then we use $\Pi_i$ to denote the pair of pants in $S$ with $C_i$ as a cuff. Let $\vec{u}_i$ be one of the feet of $\Pi_i$ on $l_i$. Then for the ubiquitous, $(R,\frac{\epsilon}{2})$-nearly even footed, integer valued measure $\mu_0$ on ${\bold \Pi}_{R,\epsilon}$ in Theorem 2.10 of \cite{LM}, we will not use the gluing satisfying the $(R,\frac{\epsilon}{2})$-nearly unit shearing condition (which glues all the cuffs), but use some other gluing, we can get the desired subsurfaces of $S'$.

To be precise, let $\Pi_{i,+}$ and $\Pi_{i,-}$ be two pants in ${\bold \Pi}_{R,\epsilon}$, such that $l_i$ and $\bar{l}_i$ are the oriented boundary components of $\Pi_{i,+}$ and $\Pi_{i,-}$ respectively. Moreover, $\Pi_{i,+}$ and $\Pi_{i,-}$ also satisfy that, one of the feet of $\Pi_{i,+}$ on $l_i$ is $\frac{\epsilon}{R}$-close to $\vec{v}_i$, and one of the feet of $\Pi_{i,-}$ on $\bar{l}_i$ is $\frac{\epsilon}{R}$-close to the $1+\pi i$ shearing of $\vec{u}_i$. Then there exists an $(R,\epsilon)$-nearly unit shearing gluing of $\mu_0$, which glues all the cuffs of the pants given by $\mu_0$, except the two cuffs of $\Pi_{i,+}$ and $\Pi_{i,-}$ corresponding with $l_i$ and $\bar{l}_i$. This gluing gives an immersed $(R,\epsilon)$-almost totally geodesic subsurface $S_i$ in $M$, with two boundary components.

Now we paste $S$ and $\{S_i\}_{i=1}^n$ together, by pasting $\Pi_i$ and $\Pi_{i,-}$ along $l_i$, to get a new immersed surface $S'$. By the construction, we have that $S'$ is a compact oriented $\pi_1$-injective $(R,\epsilon)$-almost totally geodesic subsurface in $M$, such that $L$ is its oriented boundary, and one of the feet on $l_i$ is $\frac{\epsilon}{R}$-close to $v_i$. By throwing away all the closed surface components of $S'$, we can suppose all the components of $S'$ are surfaces with boundary.
\end{proof}

\begin{rem}
In \cite{LM}, Liu and Markovic were interested in realizing a prescribed second homology class $\alpha \in H_2(M,L;\mathbb{Z})$ by an immersed connected $(R,\epsilon)$-almost totally geodesic subsurface. However, they need to take some non-zero integer multiple of $\alpha$ to construct such a realization. In this paper, we only need to find an immersed $(R,\epsilon)$-almost totally geodesic subsurface bounded by $L$, but do not care much about its homology class. So taking an integer multiple of $L$ is not necessary.
\end{rem}

\begin{rem}
Actually, for all the statements in this section, the condition
$$
\left\{ \begin{array}{l}
        |\bold{hl}(C)-\frac{R}{2}|<\epsilon, \\
        |s(C)-1|<\frac{\epsilon}{R}
\end{array} \right.
$$
can always be replaced by
$$
\left\{ \begin{array}{l}
        |\bold{hl}(C)-\frac{R}{2}|<\frac{\epsilon}{R}, \\
        |s(C)-1|<\frac{\epsilon}{R^2},
\end{array} \right.
$$
while ${\bold \Gamma}_{R,\epsilon}$ and ${\bold \Pi}_{R,\epsilon}$ can also be replace by ${\bold \Gamma}_{R,\frac{\epsilon}{R}}$ and ${\bold \Pi}_{R,\frac{\epsilon}{R}}$ respectively, as we did in \cite{Su} (pointed out in \cite{Sa}). So when we apply the results in this section, we always suppose such an $\frac{1}{R}$ factor has been multiplied to $\epsilon$. The main reason that such a refinement is applicable is, the exponential mixing property of frame flow (\cite{Mo}, \cite{Po}) gives exponential mixing rate, which beats any polynomial rate.
\end{rem}

\section{Construction of Immersed $\pi_1$-injective $2$-complex and Virtual Domination}\label{construction1}

In this section, we will prove Theorem \ref{main}. The essential step is to construct an immersed $\pi_1$-injective $2$-complex in any closed hyperbolic $3$-manifold (Section \ref{construction2}). The topology and geometry of this $2$-complex is suggested by some nice orbifold handle structures of $M_0$ and $N$, which are described in Section \ref{handle}. To highlight the geometrical and topological idea, the proof of two technical results on geometric estimations are delayed to Section \ref{injectivity}.

\subsection{Orbifold Handle Structures of $M_0$ and $N$}\label{handle}

In \cite{Th1}, Thurston described a hyperbolic $3$-orbifold $M_0$ whose underlying space is $S^3$, and the singular set is the Borromean rings with indices $4$ (cone angle $\frac{\pi}{2}$). $M_0$ can be obtained as a quotient space of the cube by the following way. For each face of the cube, there is a dark segment drawn in this face as in Figure 1. Then the quotient relation on this face is given by the reflection along the arc, and $S^3$ is the quotient space of the cube by six such reflections. The six dark segments in Figure 1 correspond with the Borromean rings in $S^3$.

The rectangles in Figure 1 are actually combinatorial pentagons, and the combinatorial structure of the boundary of the cube in Figure 1 is isomorphic to the combinatorial structure of the boundary of the regular dodecahedron. Since the hyperbolic right-angled regular dodecahedron exists, it is easy to see that $M_0$ is a hyperbolic $3$-orbifold and $\pi_1(M_0)$ is commensurable with the reflection group of the hyperbolic right-angled regular dodecahedron.

\begin{center}
\includegraphics[width=3in]{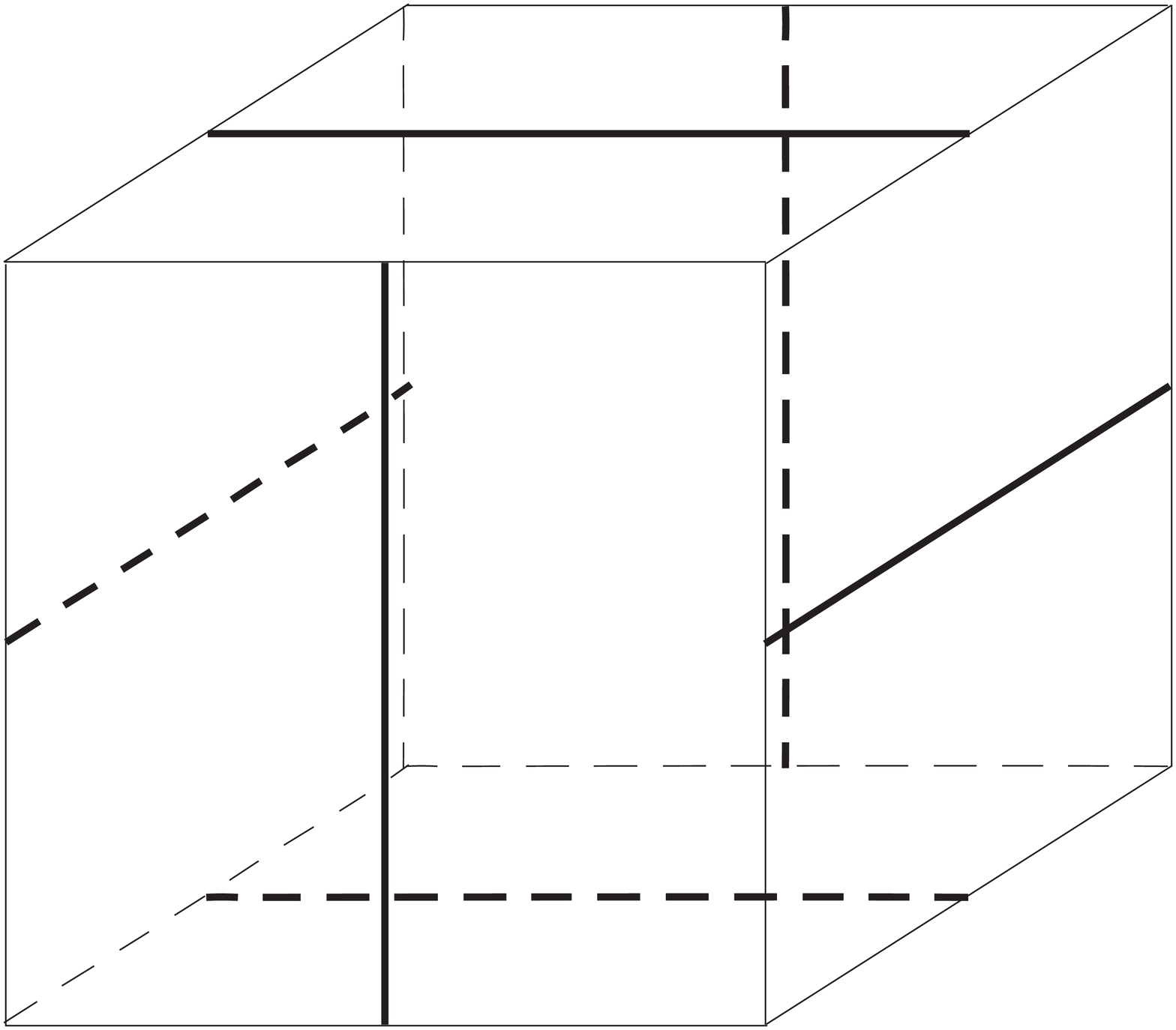}
\vskip 0.5 truecm
 \centerline{Figure 1}
\end{center}

Now we define the notion of orbifold handle structure.

\begin{defn}\label{ohs}
Suppose $P$ is a $3$-orbifold, whose underlying space is a closed $3$-manifold, and its singular set is a union of disjoint embedded circles. An {\it orbifold handle structure} of $P$ is a handle structure of the underlying space of $P$, such that each $0$- and $1$-handle does not intersect with the singular set, and each $2$- and $3$-handle intersects with the singular set at most along one arc.
\end{defn}

Note that, if $P'$ is a finite orbifold-cover of $P$, then an orbifold handle structure of $P$ lifts to an orbifold handle structure of $P'$.

In Figure 2, we show the $0$- and $1$-handles of an orbifold handle structure of $M_0$, here all the $0$- and $1$-handles do not intersect with the singular set.

\begin{center}
\includegraphics[width=3.8in]{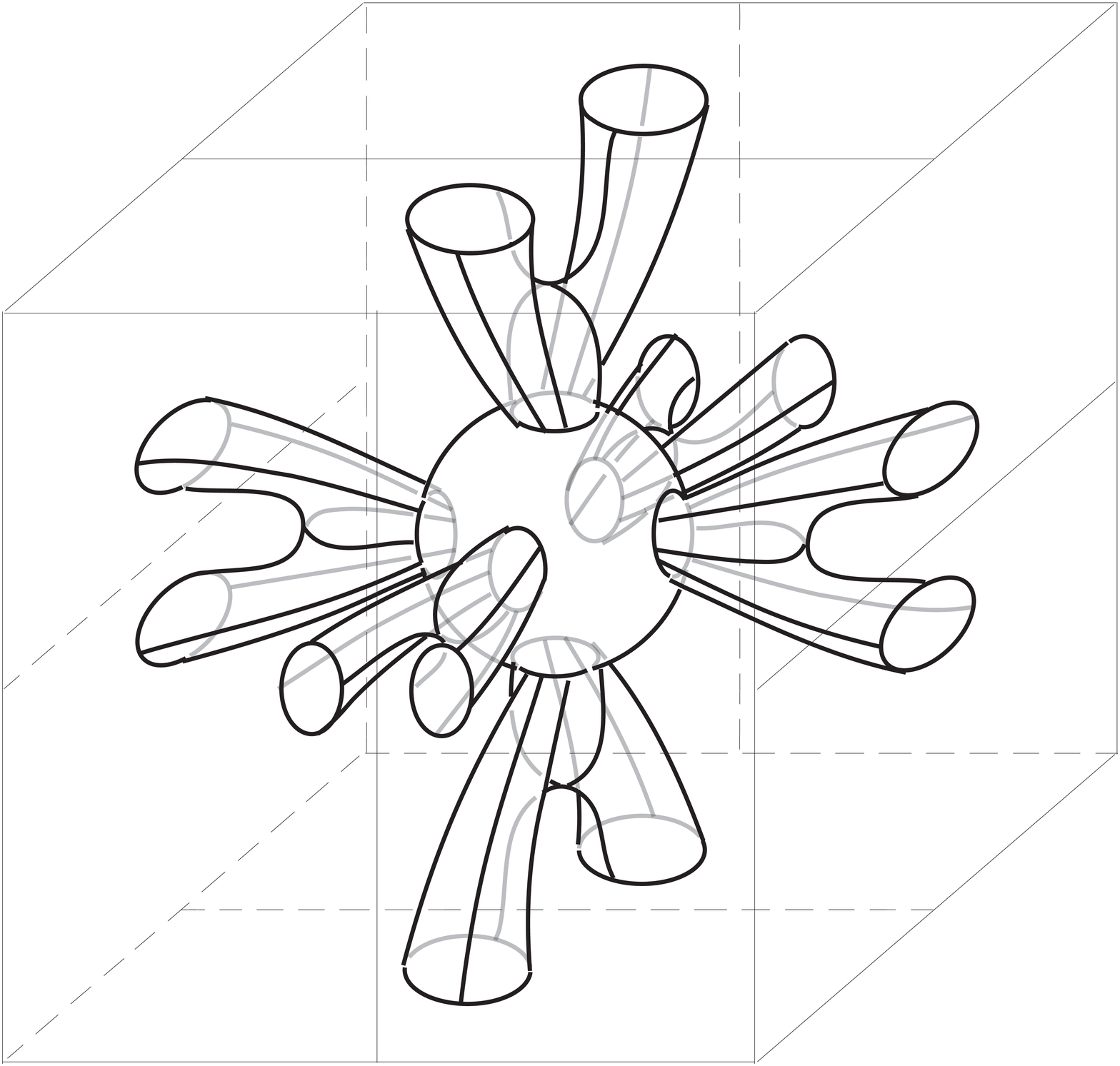}
\vskip 0.5 truecm
 \centerline{Figure 2}
\end{center}

In Figure 3, an $1$-dimensional subcomplex $T$ of $M_0^{(1)}$ is given (as union of red arcs). $T$ consists of one $0$-cell and six $1$-cells, and gives a generating set of $\pi_1(M_0)$. The generators correspond with the six oriented arcs $a,b,c,d,e,f$ in Figure 3, and it is easy to get a presentation of $\pi_1(M_0)$:
\begin{equation}\label{1}
 \pi_1(M_0)=\left \langle a,b,c,d,e,f\ \left| \begin{array}{l l}
adb^{-1}d^{-1}, acb^{-1}c^{-1}, eaf^{-1}a^{-1}, ebf^{-1}b^{-1}, ced^{-1}e^{-1}, \\cfd^{-1}f^{-1},
 a^4,b^4,c^4,d^4,e^4,f^4
 \end{array}\right. \right \rangle.
\end{equation}

Each relator in this group presentation corresponds with a $2$-handle in the orbifold handle structure of $M_0$. The first six relators correspond with six $2$-handles that do not intersect with the singular set. For the remaining six relators, each of them corresponds with a $2$-handle intersecting with the singular set at one arc. Under the hint of the group presentation \eqref{1}, the readers can try to figure out the position of the $2$-handles in Figure 2. We do not draw the whole picture here since it will be very messy if everything are shown together.

Given the $0$, $1$, and $2$-handles of $M_0$, it is easy to see that the orbifold handle structure has seven $3$-handles. One of the $3$-handle does intersect with the singular set. For the remaining six $3$-handles, each of them intersects with the singular set at one arc.

Figure 4 shows the Kirby diagram of the orbifold handle structure of $M_0$, as a handle structure of $S^3$. In Figure 4, the six letters $A,B,C,D,E,F$ and their reflections show how the corresponding $1$-handles are attached, or equivalently, how the pairs of disks are identified with each other (by reflection about horizontal or vertical lines). The letters $A,B,C,D,E,F$ in Figure 4 correspond with the six generators in the group presentation \eqref{1} of $\pi_1(M_0)$.

\begin{center}
\psfrag{a}[]{\color{red} $a$} \psfrag{b}[]{\color{red} $b$} \psfrag{c}[]{\color{red} $c$} \psfrag{d}[]{\color{red} $d$} \psfrag{e}[]{\color{red} $e$} \psfrag{f}[]{\color{red} $f$}
\includegraphics[width=3.8in]{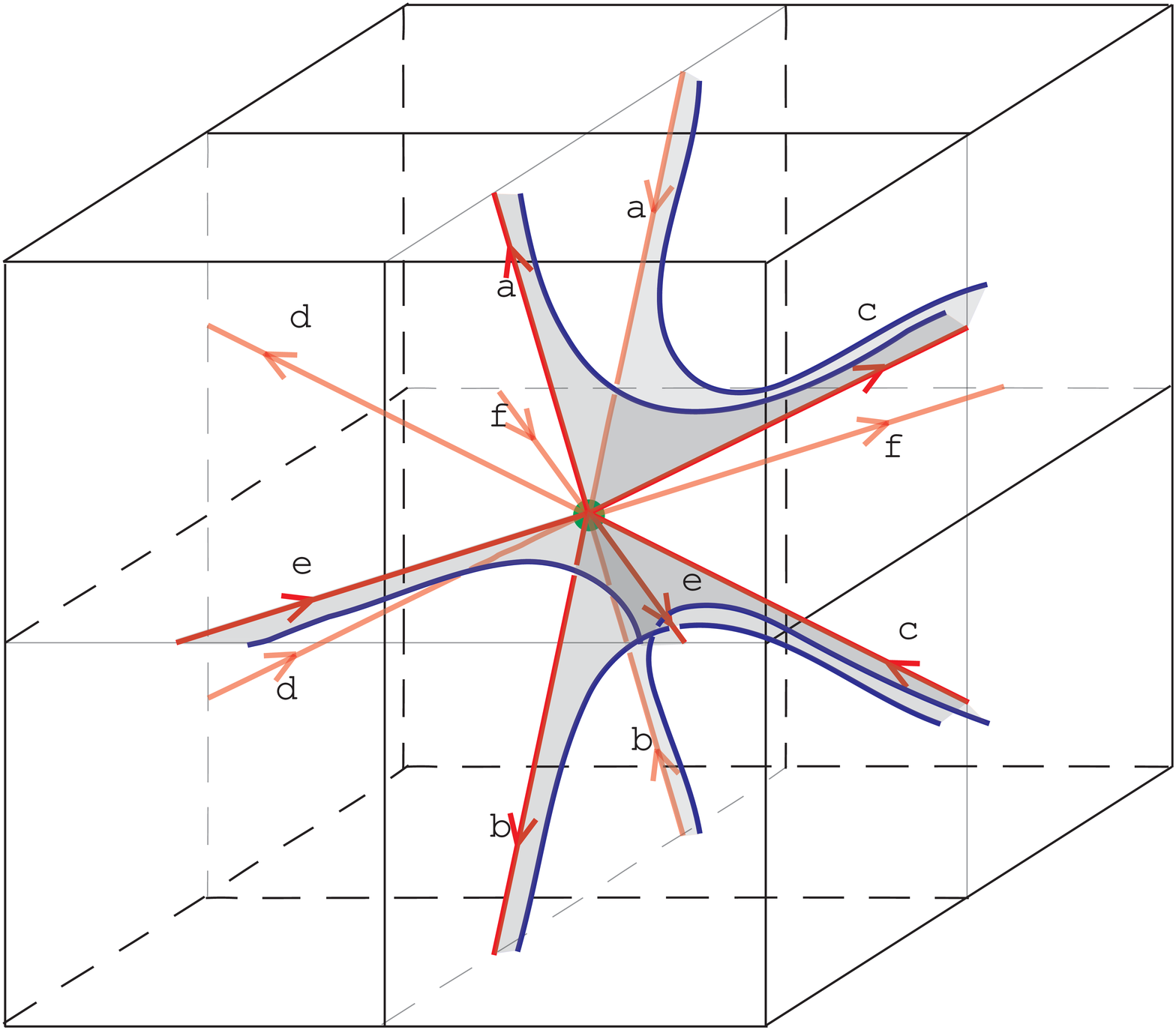}
\vskip 0.5 truecm
 \centerline{Figure 3}
\end{center}

\begin{center}
\psfrag{a}[]{$A$} \psfrag{b}[]{$B$} \psfrag{c}[]{$C$} \psfrag{d}[]{$D$} \psfrag{e}[]{$E$} \psfrag{f}[]{$F$}
\psfrag{A}[]{\rotatebox{180}{\reflectbox{$A$}}} \psfrag{B}[]{\rotatebox{180}{\reflectbox{$B$}}} \psfrag{C}[]{\reflectbox{$C$}} \psfrag{D}[]{\reflectbox{$D$}} \psfrag{E}[]{\reflectbox{$E$}} \psfrag{F}[]{\reflectbox{$F$}}
\includegraphics[width=3.6in]{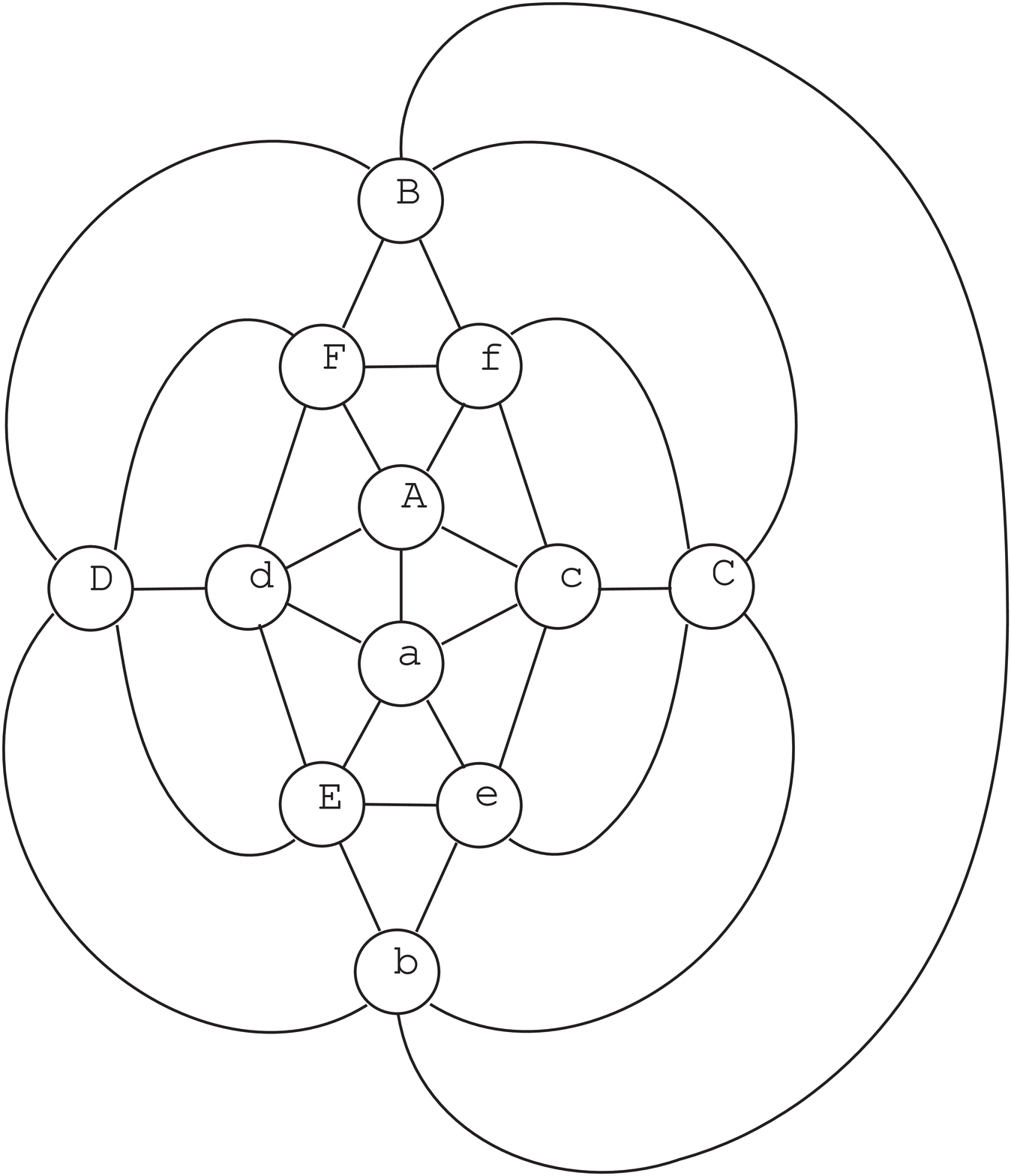}
\vskip 0.5 truecm
 \centerline{Figure 4}
\end{center}

From Figure 4, we can see that the orbifold handle structure (group presentation \eqref{1}) of $M_0$ is closely related with the geometry of the regular icosahedron. The $0$-handle corresponds with the center of the regular icosahedron. The six $1$-handles give twelve vectors at the center, which correspond with the twelve vertices of the regular icosahedron. The twelve $2$-handles correspond with the thirty edges on the boundary of the regular icosahedron. For the first six $2$-handles (correspond with the first six relators in the group presentation \eqref{1}), each of them corresponds with four edges. For the remaining six $2$-handles, each of them corresponds with one edge, since the corresponding $2$-handle intersects with the singular set (index $4$) along one arc.

Now we construct a $2$-dimensional subcomplex $X$ in $M_0$, which is a deformation retract of $M_0^{(2)}$. It is the union of $T$ and twelve topological discs which are the cores of the corresponding twelve $2$-handles of $M_0$ (with boundary as concatenation of arcs in $T$). Six of these discs are $4$-gons which do not intersect with the singular set. The remaining six are monogons, and each of them intersects with the singular set at one point. For the convenience of our further construction, we subdivide each topological disc to be the union of a cornered annulus and a round disc, such that the cornered annulus lies in $M_0^{(1)}$, and the round disc lies in a $2$-handle of $M_0$. In Figure 3, the shaded part corresponds with two such cornered annuli in $M_0$. One of them is a four-cornered annulus which corresponds with the relator $acb^{-1}c^{-1}$, and the other component is a one-cornered annulus which corresponds with the relator $e^4$. We still use $X$ to denote this $2$-complex with the refined combinatorial structure, and use $X'$ to denote the intersection of $X$ and $M_0^{(1)}$ (which excludes all the round discs). Although $X$ is not a genuine $2$-complex, it is easy to subdivide it to get a $2$-complex, so we will simply call $X$ a $2$-complex. The $0$-cell and $1$-cells in $T\subset X$ will still be called $0$-cell and $1$-cells. The round discs will be called $2$-cells, and the intersection of $2$-cells and cornered-annuli will be called circles.

In summary, $X\subset M_0$ is a $2$-complex consists of one $0$-cell, six $1$-cells, twelve circles, twelve $2$-cells (six of them intersect with the singular set), six four-cornered annuli and six one-cornered annuli. In $X'$, those twelve $2$-cells are excluded. In Figure 5, we show the picture of the four-cornered annulus and one-cornered annulus, for convenience of the readers.

\begin{center}
\includegraphics[width=3.5in]{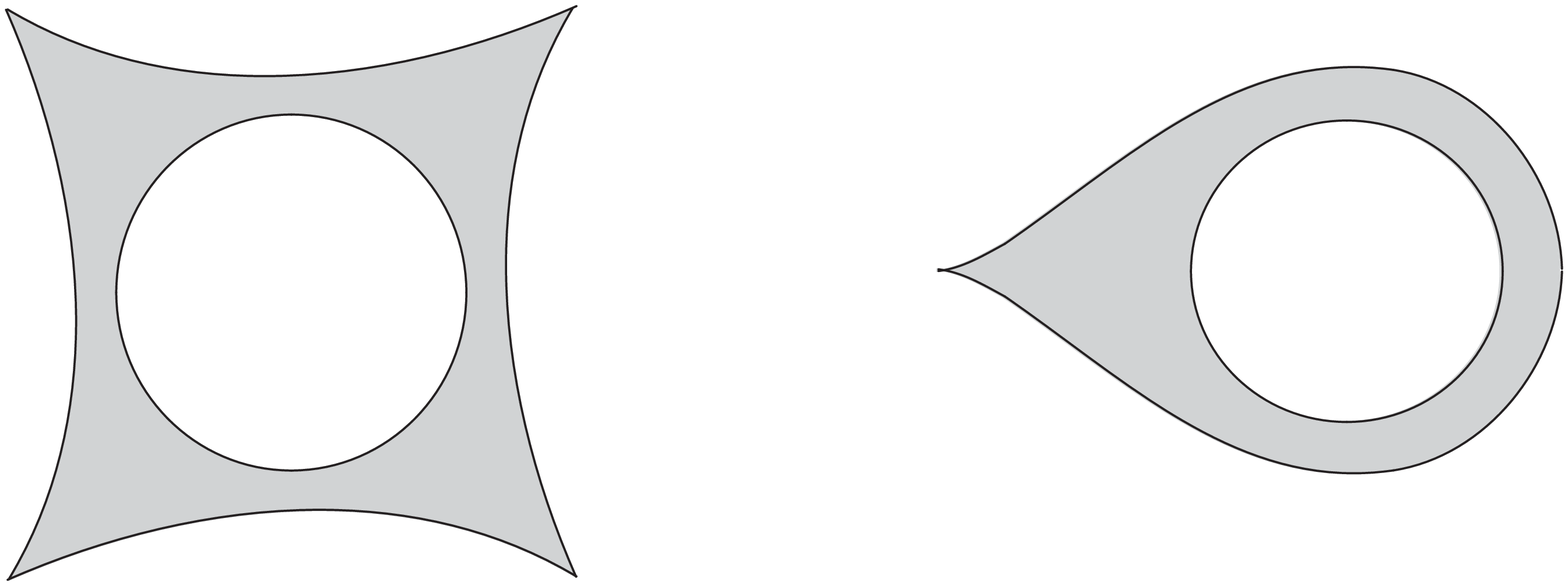}
\vskip 0.5 truecm
 \centerline{Figure 5}
\end{center}

Theorem \ref{HLMW} can be rephrased to the following statement. For any closed oriented $3$-manifold $N$, $M_0$ admits a finite orbifold cover $M_N$, whose underlying space is homeomorphic to $N$ with respect to their orientations. So the orbifold handle structure of $M_0$ lifts to an orbifold handle structure of $M_N$, then induces a handle structure of $N$. We may abuse the notation between $M_N$ and $N$, when the orbifold structure is not important in the context.

Suppose $M_N\rightarrow M_0$ is a $d$-sheet orbifold cover, then the induced handle structure of $N$ has $d$ $0$-handles, $6d$ $1$-handles and $k$ $2$-handles. Here the relation between $d$ and $k$ depends on the geometric behavior of the orbifold covering map $M_N\rightarrow M_0$ along the singular set. Actually, $\frac{15}{2}d\leq k\leq12d$ holds. Those six $2$-handles of $M_0$ that do not intersect with the singular set lift to $6d$ $2$-handles of $N$. For each of those six $2$-handles intersecting with the singular set, the number of the components of its preimage is between $\frac{d}{4}$ and $d$.

Let $Y$ be the preimage of $X$ in $N$, then the finite branched cover $Y\rightarrow X$ is also degree $d$. $Y$ has $d$ $0$-cells (correspond with $0$-handles of $N$), $6d$ $1$-cells (correspond with $1$-handles of $N$), $k$ circles, $k$ $2$-cells and $k$ cornered annuli (correspond with $2$-handles of $N$). Here the cornered annuli might be four-cornered, two-cornered or one-cornered, depending on the branched index. Moreover, all of these three possibilities do happen, by Theorem 1.1 of \cite{HLMW}.

Here $Y$ is a deformation retract of $N^{(2)}$. Let $Y'$ be the intersection of $Y$ with $N^{(1)}$, which excludes all the $2$-cells in $Y$, then $Y'\cap \partial N^{(1)}$ is the disjoint union of $k$ circles where the $2$-handles of $N$ are attached.

\subsection{Construction of Immersed $\pi_1$-injective $2$-complex}\label{construction2}
In this section, for any closed hyperbolic $3$-manifold $M$, we will construct an immersed $\pi_1$-injective $2$-complex $j:Z\looparrowright M$, where the $2$-complex $Z$ contains two copies of $Y'$ as its subcomplex.

In the following part of this paper, all the tangent vectors will be unit vectors, if there is no specific description. For any point $p\in M$, we use $T_p^1M$ to denote the set of all unit tangent vectors at $p$. For any two vectors $\vec{v}_1,\vec{v}_2 \in T_p^1M$ for some $p\in M$, we use $\Theta(\vec{v}_1,\vec{v}_2)\in [0,\pi]$ to denote the angle between $\vec{v}_1$ and $\vec{v}_2$.

At first, we need to introduce the {\it connection principle} (see Lemma 4.15 in \cite{LM}). For technical reason, we use a slightly different statement from Lemma 4.15 of \cite{LM}. It follows from the exponential mixing property of frame flow (\cite{Mo}, \cite{Po}), see \cite{Sa}.

\begin{lem}\label{connection} (Connection Principle)
For any closed hyperbolic $3$-manifold $M$ and any small number $0<\delta<1$, there exists a constant $L_0(\delta,M)>0$, such that the following statement holds. Let $\vec{t}_p,\vec{n}_p\in T_p^1M$ and $\vec{t}_q,\vec{n}_q\in T_q^1M$ be two pairs of unit orthonormal vectors at $p,q\in M$ respectively. For any real number $L>L_0(\delta,M)$, there exists an oriented geodesic arc $\gamma$ from $p$ to $q$, such that the following conditions hold.
\begin{itemize}
\item The initial and terminal tangent vectors of $\gamma$ at $p$ and $q$ are $\frac{\delta}{L}$-close to $\vec{t}_p$ and $\vec{t}_q$ respectively.
\item The length of $\gamma$ is $\frac{\delta}{L}$-close to $L$.
\item The angle between $\vec{n}_q$ and the parallel transportation of $\vec{n}_p$ to $q$ along $\gamma$ is $\frac{\delta}{L}$-close to $0$.
\end{itemize}
\end{lem}

\begin{rem}\label{log1}
We also assume that $10000 L_0(\delta,M)e^{-\frac{L_0(\delta,M)}{16}}<\delta$ and $L_0(\delta,M)>10000$ hold, for the convenience of further estimations.
\end{rem}

Now we can construct null-homologous closed geodesic arcs with a base point in any hyperbolic $3$-manifold, which is a basic step of our construction.

\begin{lem}\label{nullhomologous}
For any small number $0<\delta<1$, there exists $L_1(\delta,M)>0$, such that for any real number $L>L_1(\delta,M)$, the following statement holds. Let $p$ be a point in $M$, $\vec{v}_1,\vec{v}_2$ be two unit vectors in $T_p^1 M$, and $\vec{n}\in T_p^1 M$ be another unit vector orthogonal with both $\vec{v}_1$ and $\vec{v}_2$. Then there exists an oriented geodesic arc $\gamma$ based at $p$ such that the following conditions hold.
\begin{itemize}
\item The initial and terminal tangent vectors of $\gamma$ at $p$ are $\frac{\delta}{L}$-close to $\vec{v}_1$ and $\vec{v}_2$ respectively.
\item The length of $\gamma$ is $\frac{\delta}{L}$-close to $L$.
\item The angle between $\vec{n}$ and the parallel transportation of $\vec{n}$ to $p$ along $\gamma$ is $\frac{\delta}{L}$-close to $0$.
\item As a closed curve in $M$, $[\gamma]=0\in H_1(M;\mathbb{Z})$.
\end{itemize}
\end{lem}

\begin{proof}
Choose two auxiliary unit vectors $\vec{v}_1',\vec{v}_2'\in T_p^1M$ orthogonal with $\vec{n}$, such that the angles $\theta_1=\Theta(\vec{v}_1,\vec{v}_2')$, $\theta_2=\Theta(\vec{v}_2',-\vec{v}_1')$ and $\theta_3=\Theta(-\vec{v}_1',\vec{v}_2)$ are all in $[0,\frac{\pi}{3}]$.

Now we take $L_1(\delta,M)=4L_0(\frac{\delta}{100},M)$. By applying the connection principle (Lemma \ref{connection}), we get two oriented geodesics arcs $\alpha$ and $\beta$ in $M$ based at $p$ such that the following conditions hold.
\begin{itemize}
\item The initial and terminal tangent vectors of $\alpha$ are $\frac{\delta}{100}/\frac{L}{4}=\frac{\delta}{25L}$-close to $\vec{v}_1$ and $\vec{v}_1'$ respectively, while the initial and terminal tangent vectors of $\beta$ are $\frac{\delta}{25L}$-close to $-\vec{v}_2$ and $\vec{v}_2'$ respectively.
\item The length of $\alpha$ and $\beta$ are $\frac{\delta}{25L}$-close to $\frac{L}{4}$ and $\frac{L}{4}+\sum_{i=1}^3\log{(\sec{\frac{\theta_i}{2}})}$ respectively.
\item The angle between $\vec{n}$ and the parallel transportation of $\vec{n}$ to $p$ along $\alpha$ is $\frac{\delta}{25L}$-close to $0$, and so does the angle between $\vec{n}$ and the parallel transportation of $\vec{n}$ to $p$ along $\beta$.
\end{itemize}

Here $\log{(\sec{\frac{\theta}{2}})}$ shows up since $I(\theta)=2\log{(\sec{\frac{\theta}{2}})}$ is the inefficiency constant given by exterior angle $\theta$ (see \cite{KM2} Section 4.1).

Let $\gamma$ be the geodesic arc homotopic to $\alpha \beta \bar{\alpha}\bar{\beta}$ with respect to the base point $p$. By applying the estimations in Lemma 4.8 of \cite{LM} and Remark \ref{log1}, we have that $\gamma$ satisfies the first three conditions in the statement. The fourth condition clearly holds since $\gamma$ is a commutator in $\pi_1(M,p)$.
\end{proof}

\begin{rem}\label{log2}
By Remark \ref{log1}, it is easy to check that $10000 L_1(\delta,M)e^{-\frac{L_1(\delta,M)}{64}}<\delta$ and $L_1(\delta,M)>10^4$ holds.
\end{rem}

Let $\theta_0=\arcsin{\sqrt{\frac{5-\sqrt{5}}{10}}}$, then $\theta_0\approx 0.1762\pi \approx 31.7175^{\circ}$. Let $\vec{n}_1$ and $\vec{n}_2$ be two normal vectors of two adjacent faces of the Euclidean regular dodecahedron (pointing outside), then the angle between $\vec{n}_1$ and $\vec{n}_2$ equals $2\theta_0$.

Now we can do the first step our construction. For any closed oriented hyperbolic $3$-manifold $M$, take an arbitrary point $p\in M$, we choose an orthonormal frame $(\vec{e}_1,\vec{e}_2,\vec{e}_3)$ in $T_pM$ such that $(\vec{e}_1,\vec{e}_2,\vec{e}_3)$ coincide with the orientation of $M$ (i.e. $\vec{e}_3=\vec{e}_1\times \vec{e}_2$). Then for any $\vec{v}\in T_pM$, $\vec{v}$ can be written as a linear combination of $\vec{e}_1$, $\vec{e}_2$ and $\vec{e}_3$, as $\vec{v}=v_1\vec{e}_1+v_2\vec{e}_2+v_3\vec{e}_3$. In this case, we will use $\vec{v}=(v_1,v_2,v_3)$ to denote the coordinate of $\vec{v}$.

Let
$$\begin{array}{l l}
\vec{v}_a=(\sin{\theta_0},0,\cos{\theta_0}), & \vec{v}_b=(\sin{\theta_0},0,-\cos{\theta_0}),\\
\vec{v}_c=(0,\cos{\theta_0},\sin{\theta_0}), & \vec{v}_d=(0,-\cos{\theta_0},\sin{\theta_0}),\\
\vec{v}_e=(\cos{\theta_0},\sin{\theta_0},0), & \vec{v}_f=(-\cos{\theta_0},\sin{\theta_0},0)
\end{array}$$
be six unit vectors in $T_p^1M$. These six vectors together with the negative vectors of them form the twelve normal vectors of the twelve faces of an Euclidean regular dodecahedron in $T_pM$. Figure 3 shows a picture of these vectors, with $\vec{e}_1$ normal to the front face of the cube, $\vec{e}_2$ normal to the right face, and $\vec{e}_3$ normal to the top face, all these three vectors are pointing outside of the cube.

Now we apply Lemma \ref{nullhomologous} to construct six closed oriented geodesic arcs $a,b,c,d,e,f$ based at $p$. They correspond with the six generators in the group presentation \eqref{1} (we abuse the notation between generators and geodesic arcs based at $p$). For our application of Lemma \ref{nullhomologous}, the constant $\delta>0$ will be some very small number and $L>L_1(\delta,M)$ will be some very large number which will be determined later. These six geodesic arcs are constructed according to the data in Table 1. Here $\vec{v}_1$, $\vec{v}_2$ and $\vec{n}$ in Table 1 are the input data of Lemma \ref{nullhomologous}, and geodesic arcs $a,b,c,d,e,f$ are the output data. $\vec{n}$ is called an {\it almost normal vector} of the corresponding geodesic arc at its initial and terminal point.

\begin{table}[!h]
\begin{center}
\begin{tabular}{|c|c|c|c|}
  \hline
  geodesic arc & $\vec{v}_1$ & $\vec{v}_2$ & $\vec{n}$\\
   \hline
  $a$ & $\vec{v}_a$ & $\vec{v}_b$ & $(0,1,0)$\\
  \hline
  $b$ & $\vec{v}_b$ & $\vec{v}_a$ & $(0,1,0)$\\
  \hline
  $c$ & $\vec{v}_c$ & $\vec{v}_d$ & $(1,0,0)$\\
   \hline
  $d$ & $\vec{v}_d$ & $\vec{v}_c$ & $(1,0,0)$\\
  \hline
  $e$ & $\vec{v}_e$ & $\vec{v}_f$ & $(0,0,1)$\\
  \hline
  $f$ & $\vec{v}_f$ & $\vec{v}_e$ & $(0,0,1)$\\
  \hline
\end{tabular}
\end{center}
\caption{Construction of geodesic arcs.}
\end{table}

Given this construction, we have the following estimation for closed geodesics in $M$. These closed geodesics are related with the words in the group presentation \eqref{1}.

\begin{lem}\label{Repsilon}
For any positive integer $n$, the concatenation of $n$ geodesic arcs corresponding with $a^n$ is homotopic to a null-homologous closed geodesic in $M$, with complex length $\frac{25n\delta}{L}$-close to $n(L-2\log{\csc{\theta_0}})$. The same statement also holds for $b,c,d,e,f$.

For each of the twelve relators in the group presentation \eqref{1}, the concatenation of the four corresponding geodesic arcs is homotopic to a null-homologous closed geodesic in $M$, with complex length $\frac{100\delta}{L}$-close to $4(L-2\log{\csc{\theta_0}})$.
\end{lem}

\begin{proof}
We will only prove the lemma for $a^n$ and $acb^{-1}c^{-1}$, and the proof for the other words are exactly the same. These closed geodesics are clearly null-homologous in $M$, since all the geodesic arcs $a,b,c,d,e,f$ based at $p$ are null-homologous, by Lemma \ref{nullhomologous}.

{\bf 1) Proof for $a^n$:} By the construction of the geodesic arc $a$, the initial and terminal tangent vectors of $a$ at $p$ are $\frac{\delta}{L}$-close to $\vec{v}_a$ and $\vec{v}_b$  respectively, and the parallel transportation of $(0,1,0)$ along $a$ is $\frac{\delta}{L}$-close to $(0,1,0)$.

Then Lemma 4.8 of \cite{LM} and Remark \ref{log2} imply the desired estimation.

{\bf 2) Proof for $acb^{-1}c^{-1}$:} In this case, the almost normal vectors in Table 1 of adjacent oriented geodesic arcs, say $a$ and $c$, do not coincide. So we need to choose new almost normal vectors at the initial and terminal points of $a$, $c$, $b^{-1}$ and $c^{-1}$ respectively, such that adjacent geodesic arcs share the same almost normal vector at their intersection point.

Since the parallel transportation of $\vec{v}_a=(\sin{\theta_0},0,\cos{\theta_0})$ and $(0,1,0)$ to $p$ along $a$ are $\frac{\delta}{L}$-close to $\vec{v}_b=(\sin{\theta_0},0,-\cos{\theta_0})$ and $(0,1,0)$ respectively, the parallel transportation of $(\frac{\sqrt{5}+1}{4},\frac{\sqrt{5}-1}{4}, -\frac{1}{2})$ to $p$ along $a$ is $\frac{4\delta}{L}$-close to $(-\frac{\sqrt{5}+1}{4},\frac{\sqrt{5}-1}{4}, -\frac{1}{2})$. This estimation holds because $$(\frac{\sqrt{5}+1}{4},\frac{\sqrt{5}-1}{4}, -\frac{1}{2})=\sqrt{\frac{5+\sqrt{5}}{8}}(0,1,0)\times \vec{v}_a+\frac{\sqrt{5}-1}{4}(0,1,0)$$ and $$(-\frac{\sqrt{5}+1}{4},\frac{\sqrt{5}-1}{4}, -\frac{1}{2})=\sqrt{\frac{5+\sqrt{5}}{8}}(0,1,0)\times \vec{v}_b+\frac{\sqrt{5}-1}{4}(0,1,0).$$

Here $(\frac{\sqrt{5}+1}{4},\frac{\sqrt{5}-1}{4}, -\frac{1}{2})$ is orthogonal with both $-\vec{v}_c$ and $\vec{v}_a$, while $(-\frac{\sqrt{5}+1}{4},\frac{\sqrt{5}-1}{4}, -\frac{1}{2})$ is orthogonal with both $\vec{v}_b$ and $\vec{v}_c$.

By the same argument, applied to $c$, $b^{-1}$ and $c^{-1}$ respectively, we have the following estimations.
\begin{itemize}
\item The parallel transportation of $(-\frac{\sqrt{5}+1}{4},\frac{\sqrt{5}-1}{4}, -\frac{1}{2})$ to $p$ along $c$ is $\frac{4\delta}{L}$-close to\\ $(-\frac{\sqrt{5}+1}{4},\frac{\sqrt{5}-1}{4}, \frac{1}{2})$.
\item The parallel transportation of $(-\frac{\sqrt{5}+1}{4},\frac{\sqrt{5}-1}{4}, \frac{1}{2})$ to $p$ along $b^{-1}$ is $\frac{4\delta}{L}$-close to \\ $(\frac{\sqrt{5}+1}{4},\frac{\sqrt{5}-1}{4}, \frac{1}{2})$.
\item The parallel transportation of $(\frac{\sqrt{5}+1}{4},\frac{\sqrt{5}-1}{4}, \frac{1}{2})$ to $p$ along $c^{-1}$ is $\frac{4\delta}{L}$-close to \\ $(\frac{\sqrt{5}+1}{4},\frac{\sqrt{5}-1}{4}, -\frac{1}{2})$.
\end{itemize}

Now for each geodesic arc of the piecewise geodesic path $acb^{-1}c^{-1}$, it is equipped with new almost normal vectors at the initial and terminal points, such that adjacent oriented arcs share the same almost normal vector at their intersection point. So we can apply Lemma 4.8 of \cite{LM} and Remark \ref{log2} again to get the desired estimation.
\end{proof}

\begin{rem}
For any two adjacent oriented geodesic arcs in $acb^{-1}c^{-1}$, the tangent vector of the second geodesic arc at its initial point is $\frac{2\delta}{L}$-close with the $\pi-2\theta_0$ clockwise rotation (about the common almost normal vector in Lemma \ref{Repsilon}) of the tangent vector of the first geodesic arc at its terminal point. This property guarantees that the corresponding four-cornered annulus in $X'$ can be immersed into $M$ with an almost totally geodesic subsurface as its image, such that one boundary component of the annulus is mapped to the piecewise geodesic  closed path $acb^{-1}c^{-1}$, and the other boundary component is mapped to the corresponding closed geodesic in $M$.

The readers can imagine the case that the picture appears exactly on a hyperbolic surface. It those four "turns" are not all left turn (or all right turn), the four-cornered annulus can not be immersed into the hyperbolic surface, since the piecewise geodesic closed path and the closed geodesic may intersect with each other.
\end{rem}

Now let $\delta=\frac{\epsilon}{400}$ and $L=\cosh^{-1}{(\frac{\cosh{R}+\cos^2{\theta_0}}{\sin^2{\theta_0}})}=R+2\log{(\csc{\theta_0})}+O(e^{-R})$ as our input data. Here $\epsilon>0$ is so small and $R>0$ is so large, such that Corollary \ref{bound}, Theorem \ref{quasi1} and Theorem \ref{technical} hold for $(R,\epsilon)$. Moreover, we also require that Lemma \ref{nullhomologous} and Remark \ref{log2} hold for $(L,\delta)$. Then Lemma \ref{Repsilon} implies that all those twelve closed geodesics in $M$ corresponding with the twelve relators in the group presentation \eqref{1} lie in ${\bold \Gamma}_{4R,\frac{\epsilon}{4R}}$.

Now we are ready to construct an immersed $\pi_1$-injective $2$-complex $j:Z\looparrowright M$, under the hint of the handle structure of $N$. Here we divide the construction into a few steps.

{\bf Step I.} Recall that, by the end of Section \ref{handle}, we have constructed a $2$-complex $X'$ in $M_0$, which consists of one $0$-cell, six $1$-cells, twelve circles, six four-cornered annuli and six one-cornered annuli. $X'$ lifts to a $2$-complex $Y'$ in $N$, here $Y'$ a $d$-sheet cover of $X'$ and is a deformation retract of $N^{(1)}$. Moreover, $Y'\cap \partial N^{(1)}$ is exactly the union of $k$ disjoint circles on $\partial N^{(1)}$ where the $2$-handles of $N$ are attached.

We can construct an immersion $j':X'\looparrowright M$ for any closed hyperbolic $3$-manifold $M$ by the following way.

Let the $0$-cell of $X'$ be mapped to the point $p \in M$ which we have already chosen. Let the six oriented $1$-cells of $X'$ (correspond with the six generators $a,b,c,d,e,f$ in the group presentation \eqref{1}) be mapped to the six corresponding oriented geodesic arcs in $M$ based at $p$, by applying Lemma \ref{nullhomologous} to the data in Table 1. Let the twelve circles in $X'$ be mapped to twelve closed geodesics in $M$, which are homotopic to the corresponding concatenation of geodesic arcs. Six of them correspond with the first six relators in the group presentation \eqref{1}, and the remaining six of them correspond with a quarter of the remaining six relators. For example, one of them is mapped to the closed geodesic in $M$ which is homotopic to the closed geodesic arc $e$, instead of $e^4$. Lemma \ref{Repsilon} implies that the first six closed geodesics lie in ${\bold \Gamma}_{4R,\frac{\epsilon}{4R}}$, and the remaining six lie in ${\bold \Gamma}_{R,\frac{\epsilon}{R}}$.

For the four-cornered annuli (one-cornered annuli) in $X'$, we may subdivide them in the following way. First add an arc from each corner point to its opposite circle, to divide the annuli into four (one) $4$-gons. Then add a diagonal to each $4$-gon, which divide the $4$-gon to two triangles.

Then the arcs from the corner points to the opposite circle are mapped to the geodesic arcs in the right homotopic class with initial point $p$ and perpendicular with the corresponding closed geodesic in $M$. For the remaining part of the cornered annuli, $j'$ maps arcs to geodesic arcs, and map triangles to totally geodesic triangles.

{\bf Step II.} $j':X'\looparrowright M$ induces an immersion $j'': Y'\looparrowright M$, by take the composition of $j'$ and the finite sheet cover from $Y'$ to $X'$. $Y'$ consists of $d$ $0$-cells, $6d$ $1$-cells, $k$ circles and $k$ cornered annuli. Moreover, Theorem \ref{HLMW} and Lemma \ref{Repsilon} imply that all the circles in $Y'$ are mapped to closed geodesics in ${\bold \Gamma}_{4R,\frac{\epsilon}{4R}}$, ${\bold \Gamma}_{2R,\frac{\epsilon}{2R}}$ or ${\bold \Gamma}_{R,\frac{\epsilon}{R}}$. We denote these circles in $Y'$ by $C_1,\cdots,C_k$ and their images under $j''$ by $\gamma_1,\cdots,\gamma_{k}$. For each $C_i$, we give it an arbitrary orientation, and it induces an orientation on $\gamma_i$. Lemma \ref{Repsilon} also implies that $\gamma_i$ is null-homologous in $M$ for $i=1,\cdots,k$.

{\bf Step III.} For each circle $C_i\subset Y'$ which lies in a cornered annulus, in the subdivision of $Y'$ induced by the subdivision of $X'$, there are a few (one, two or four) arcs from the corner points to $C_i$. By the construction of $j':X'\looparrowright M$, $j''$ maps these arcs to geodesic arcs perpendicular with $\gamma_i$, and we choose one such geodesic arc for each $\gamma_i$, and denote it by $\alpha_i$. We denote the intersection between $j''(\alpha_i)$ and $\gamma_i$ by $p_i$, and denote the unit tangent vector of $j''(\alpha_i)$ at $p_i$ by $\vec{u}_i$ (pointing to $p_i$). Let $\vec{v}_i$ be the unit normal vector of $\gamma_i$ which is the $1+\pi i$ shearing of $\vec{u}_i$ along $\gamma_i$.

Now we apply Corollary \ref{bound} to $L_i=\gamma_i \sqcup \gamma_i$, the union of two copies of $\gamma_i$, then $L_i\in \mathbb{Z}{\bold \Gamma}_{r_iR,\frac{\epsilon}{r_iR}}$ for $r_i\in\{1,2,4\}$. Here $C_i\subset Y'$ is one boundary component of an $r_i$-cornered annulus. Since $\sigma(\gamma_i)\in \mathbb{Z}_2$, we have $\sigma(L_i)=\bar{0}\in \mathbb{Z}_2$. So Corollary \ref{bound} implies that $L_i$ bounds an immersed oriented $(r_i R,\frac{\epsilon}{r_i R})$-almost totally geodesic subsurface $f_i: S_i\looparrowright M$, such that the two pair of pants containing the two boundary components of $S_i$ both have a foot  $\frac{\epsilon}{(r_i R)^2}$-close to $\vec{v}_i$. Since we can suppose $S_i$ does not have closed surfaces as it connected components, there are two possibilities for $S_i$. $S_i$ is either a connected surface with two boundary components, or the union of two connected surfaces, and each of them has one boundary component. Actually, the second possibility does not always happen. For example, for the closed geodesic $\gamma_j$ corresponding with $e^4$ (which exists by Theorem 1.1 of \cite{HLMW}), its canonical lifting goes around some closed curve in $SO(M)$ for four times, then do a $2\pi$-counterclockwise rotation about some vector. So $\sigma(\gamma_j)=\bar{1}$, and Corollary \ref{bound} implies that $\gamma_j$ itself does not bound an immersed $(4R,\frac{\epsilon}{4R})$-almost totally geodesic subsurface in $M$. So $S_j$ is a connected surface in this case.

{\bf Step IV.} By applying the surgery argument as in Section 2.1 of \cite{Su}, we can suppose each $S_i$ satisfies the following property. Endow $S_i$ with the hyperbolic metric such that each curve in the pants decomposition of $S_i$ has length exactly $r_iR$, and the shearing parameter is exactly $1$, then any essential arc $(I,\partial I)\rightarrow (S_i,\partial S_i)$ has length greater than $\frac{R}{2}$.

Now we can define our desired $2$-complex $Z$, which is the quotient space of the union of two copies of $Y'$ (denoted by $Y'_1$ and $Y'_2$) and $\{S_i\}_{i=1}^{k}$. For each $S_i$, the two boundary components of $S_i$ are pasted to the two copies of $C_i$ in $Y'_1$ and $Y'_2$ respectively, by orientation preserving homeomorphism. Then $j'':Y'\looparrowright M$ and $f_i: S_i\looparrowright M$ give the desired immersion $j:Z\looparrowright M$, with induced map $j_*:\pi_1(Z)\rightarrow \pi_1(M)$. In the construction of the desired degree-$2$ map, the union of $Y_1'$ and $Y_2'$ will correspond with $N^{(1)}$, and the $S_i$'s will correspond with the $2$-handles of $N$.

Let $\rho_0:\pi_1(Z)\rightarrow PSL_2(\mathbb{C})$ be the representation such that all the $\frac{\epsilon}{R}$- and $\frac{\epsilon}{R^2}$-closeness statements appeared in $j_*:\pi_1(Z)\rightarrow \pi_1(M)\subset PSL_2(\mathbb{C})$ are replaced by exact ($0$-closeness) statements. As an algebraic object, $\rho_0: \pi_1(Z)\rightarrow PSL_2(\mathbb{C})$ is accompanied with a map $i_0:\widetilde{Z}\rightarrow \mathbb{H}^3$ from the universal cover $p:\widetilde{Z}\rightarrow Z$ of $Z$ to $\mathbb{H}^3$. The map $i_0:\widetilde{Z}\rightarrow \mathbb{H}^3$ is a geometric object which carries the algebraic information of $\rho_0: \pi_1(Z)\rightarrow PSL_2(\mathbb{C})$. The image $i_0(\widetilde{Z})$ is a union of totally geodesic subsurfaces in $\mathbb{H}^3$, pasted along a union of geodesic arcs. More precisely, $i_0:\widetilde{Z}\rightarrow \mathbb{H}^3$ is defined by the following conditions, which also give the precise definition of $\rho_0:\pi_1(Z)\rightarrow PSL_2(\mathbb{C})$.

\begin{construction}\label{standard}
\begin{itemize}
\item For each $0$-cell $x$ in $\widetilde{Z}$, it is associated with an orthonormal frame $\{\vec{e}_1,\vec{e}_2,\vec{e}_3\}$ at $i_0(x)$ with respect to the orientation of $\mathbb{H}^3$.
\item For each $1$-cell $t$ in $\widetilde{Z}$, it is endowed with an orientation such that the oriented $1$-cell $p(t)$ in $Z$ corresponds with one of the oriented $1$-cell $a,b,c,d,e,f$ in $X'$. Suppose $t$ travels from one $0$-cell $x$ to another $0$-cell $y$ in $\widetilde{Z}$, let $\{\vec{e}_1,\vec{e}_2,\vec{e}_3\}$ and $\{\vec{e}_1',\vec{e}_2',\vec{e}_3'\}$ be the two orthonormal frames at $i_0(x)$ and $i_0(y)$ respectively, then the following conditions hold.
\begin{itemize}
\item $i_0(t)$ is a geodesic arc from $i_0(x)$ to $i_0(y)$ with length equal to $L$.
\item The tangent vectors of $i_0(t)$ at $i_0(x)$ and $i_0(y)$ are exactly the corresponding vectors $\vec{v}_1$ and $\vec{v}_2$ in Table 1 respectively (under the coordinate given by orthonormal frames $\{\vec{e}_1,\vec{e}_2,\vec{e}_3\}$ and $\{\vec{e}_1',\vec{e}_2',\vec{e}_3'\}$).
\item The parallel transportation of $\{\vec{e}_1,\vec{e}_2,\vec{e}_3\}$ to $i_0(y)$ along $i_0(t)$ is equal to the counterclockwise $\pi-2\theta_0$ rotation of $\{\vec{e}_1',\vec{e}_2',\vec{e}_3'\}$ about a vector $\vec{n} \in T_{i_0(q)}^1(M)$. Here
\begin{equation}\label{2}
\vec{n}=\left\{ \begin{array}{ll}
        \vec{e}_2' & \text{if}\  p(t)\ \text{corresponds with}\ a\ \text{in}\ X', \\
        -\vec{e}_2' & \text{if}\  p(t)\ \text{corresponds with}\ b\ \text{in}\ X', \\
        \vec{e}_1' & \text{if}\  p(t)\ \text{corresponds with}\ c\ \text{in}\ X', \\
        -\vec{e}_1' & \text{if}\  p(t)\ \text{corresponds with}\ d\ \text{in}\ X', \\
        \vec{e}_3' & \text{if}\  p(t)\ \text{corresponds with}\ e\ \text{in}\ X', \\
        -\vec{e}_3' & \text{if}\  p(t)\ \text{corresponds with}\ f\ \text{in}\ X'. \\
        \end{array} \right.
\end{equation}
\end{itemize}
\item For any component of the preimage of a circle in the pants decomposition of some $S_i$, it is mapped to a bi-infinite geodesic in $\mathbb{H}^3$.
\item For any component of the preimage of a four-cornered annulus in $Z$, it is mapped to a totally geodesic subsurface in $\mathbb{H}^3$ with two boundary components. One of its boundary component is a concatenation of geodesic arcs of length $L$, and the other boundary component is a bi-infinite geodesic. Moreover, these two boundary components share the same limit points on $\partial \mathbb{H}^3=S^2_{\infty}$.
\item For any arc in $\widetilde{Z}$ which is a component of the preimage of either a seam in a pair of pants, or an arc in a cornered annulus from a corner point to the opposite circle, its image under $i_0$ is a geodesic arc perpendicular with the corresponding bi-infinite geodesics.
\item For any circle $C$ in the oriented surface $S_i$ which is the boundary of a pair of pants $\Pi$, we have ${\bf hl}_{\Pi}(C)=\frac{r_iR}{2}$. Here ${\bf hl}_{\Pi}(C)=\frac{r_iR}{2}$ means the following statement holds. Take any bi-infinite line $\beta\subset \widetilde{Z}$ which is a component of $p^{-1}(C)$, and any two arcs $\alpha$ and $\alpha'$ in $\widetilde{Z}$ that intersect with $\beta$ and project to two seams of $\Pi\subset S_i$, such that there is no other such arc between $\alpha$ and $\alpha'$ along $\beta$. Then the tangent vector of $i_0(\alpha)$ is the $\frac{r_iR}{2}$ translation of the tangent vector of $i_0(\alpha')$ along $i_0(\beta)$ (under proper orientation).
\item For any oriented circle $C$ shared by two pants $\Pi$ and $\Pi'$ in $S_i$, such that $\Pi$ lies to the left of $C$, we have $s(C)=1$. Here $s(C)=1$ means the following statement holds. Take any bi-infinite line $\beta\subset \widetilde{Z}$ which is a component of $p^{-1}(C)$, and any arc $\alpha$ in $\widetilde{Z}$ that intersects with $\beta$ and projects to a seam of $\Pi$. Let $\alpha'$ be the arc that intersects with $\beta$ and projects to a seam of $\Pi'$, such that $i_0(\alpha')$ is the nearest such arc from $i_0(\alpha)$. Then the tangent vector of $i_0(\alpha')$ is the $1+\pi i$ translation of the tangent vector of $i_0(\alpha)$ along $i_0(\beta)$ .
\item For any oriented boundary component $C_i$ of $S_i$ which is the cuff of a pair of pants $\Pi\subset S_i$, in step III of the construction of $j:Z\rightarrow M$, we have chosen an arc $\alpha_i$ which travels from some $0$-cell of $Z$ to $C_i$. Take any bi-infinite line $\beta\subset \widetilde{Z}$ which is a component of $p^{-1}(C_i)$, and any arc $\alpha$ in $\widetilde{Z}$ that intersects with $\beta$ and projects to $\alpha_i$. Let $\alpha'$ be the arc that intersects with $\beta$ and projects to a seam of $\Pi$, such that $i_0(\alpha')$ is the nearest such arc from $i_0(\alpha)$. Then the tangent vector of $i_0(\alpha')$ is the $1+\pi i$ translation of the tangent vector of $i_0(\alpha)$ along $i_0(\beta)$.
\item For any component of the preimage of a pair of pants in $S_i$, it is mapped to a totally geodesic subsurface in $\mathbb{H}^3$.
\end{itemize}
\end{construction}

Then $\rho_0:\pi_1(Z)\rightarrow PSL_2(\mathbb{C})$ can be defined by: for any $z\in \widetilde{Z}$, and any $g\in \pi_1(Z)$, $\rho_0(g)(i_0(z))=i_0(g(z))$ holds. In this case, the frames at the images of $0$-cells are also $\pi_1(Z)$-equivariant.

Although this construction is very long, it is simply describing a $2$-complex $\widetilde{Z}$ in $\mathbb{H}^3$ which is the union of a few totally geodesic subsurfaces, along an $1$-dimensional complex consisting of long geodesic arcs. The map $i_0:\widetilde{Z}\rightarrow \mathbb{H}^3$ will serve as our standard model of its algebraic counterpart $\rho_0:\pi_1(Z)\rightarrow PSL_2(\mathbb{C})$. We will study geometric properties of the deformations of $i_0$, which gives us algebraic properties of small deformations of $\rho_0$, in particular $j_*:\pi_1(Z)\rightarrow \pi_1(M)\subset PSL_2(\mathbb{C})$.

The map $i_0:\widetilde{Z}\rightarrow \mathbb{H}^3$ induces a path metric on $\widetilde{Z}$, which is denoted by $d_0$. Then an argument in hyperbolic geometry gives the following result, which will be shown in Section \ref{injectivity}.

\begin{prop}\label{quasi}
When $R$ is large enough, $i_0:(\widetilde{Z},d_0) \rightarrow (\mathbb{H}^3, d_{\mathbb{H}^3})$ is a embedding and also a quasi-isometric embedding. In particular, $\rho_0:\pi_1(Z)\rightarrow PSL_2(\mathbb{C})$ is an injective map.
\end{prop}

The following result claims that a small deformation of $\rho_0$ is still $\pi_1$-injective. It is shown by geometric estimation on the associated map $i:\widetilde{Z}\rightarrow \mathbb{H}^3$ of $j_*:\pi_1(Z)\rightarrow PSL_2(\mathbb{C})$.

\begin{thm}\label{claim}
When $\epsilon>0$ is small enough and $R>0$ is large enough, the map $j:Z\looparrowright M$ induces an injective map $j_*:\pi_1(Z)\hookrightarrow \pi_1(M)$, and $j_*(\pi_1(Z))\subset \pi_1(M)$ is a geometric finite subgroup. Moreover, $\mathbb{H}^3/j_*(\pi_1(Z))$ is homeomorphic to $\mathbb{H}^3/\rho_0(\pi_1(Z))$ with respect to their orientations.
\end{thm}

The intuitive reason of Theorem \ref{claim} is quite simple. Since $j_*$ is a small deformation of $\rho_0$, $j_*(\pi_1(Z))$ should share geometrical and topological properties with $\rho_0(\pi_1(Z))$ (as in \cite{KM1} and \cite{Su}). We will give more precise description of the deformations of $\rho_0:\pi_1(Z)\rightarrow PSL_2(\mathbb{C})$ in Section \ref{injectivity}, and the proof of Theorem \ref{claim} will also be delayed to Section \ref{injectivity}. The main reason is, although this result is very intuitive, the geometric estimation is technical. We do not want these technical stuff to distract the readers from the main idea of the proof of Theorem \ref{main}.

\subsection{Construction of Domination}\label{pinching}

In this section, we will prove Theorem \ref{main}, by assuming Proposition \ref{quasi} and Theorem \ref{claim}. Theorem \ref{main} claims that, for any closed oriented hyperbolic $3$-manifold $M$, and any closed oriented $3$-manifold $N$, $M$ admits a finite cover $M'$, such that there exists a degree-$2$ map $f:M'\rightarrow N$.

To prove Theorem \ref{main}, we first need to figure out the topology of $\mathbb{H}^3/\rho_0(\pi_1(Z))$.

Let $H$ be the oriented handlebody which is homeomorphic to $N^{(1)}$ (the union of $0$- and $1$-handles of $N$), and we fix such an orientation-preserving identification. There are $k$ disjoint annuli $A_1,\cdots,A_{k}$ in $\partial N^{(1)}$, where the $k$ $2$-handles of $N$ are attached. We abuse notation and still use $A_1,\cdots,A_{k}$ to denote the corresponding $k$ disjoint annuli on $\partial H$. Now we take two copies of $H$ and denote them by $H_1$ and $H_2$. We also denote the corresponding annuli of $A_i$ on $\partial H_1$ and $\partial H_2$ by $A_{i,1}$ and $A_{i,2}$ respectively. Here $H_j$ and $A_{i,j}$ are all endowed with the identical orientations as $H$ and $A_i$.

Now we take $k$ $I$-bundles over surfaces: $S_1\times I,\cdots,S_{k}\times I$, these $S_i$'s are the oriented surfaces constructed in step III of the construction of $Z$. Since the boundary of $S_i$ has two oriented components $C_{i,1}$ and $C_{i,2}$, the boundary of $S_i\times I$ contains two disjoint annuli $C_{i,1}\times I$ and $C_{i,2}\times I$, such that the product orientation on $C_{i,1}\times I$ and $C_{i,2}\times I$ coincide with their induced orientation from $S_i\times I$.

Let $j_i:A_{i,1}\rightarrow A_{i,2}$ be the orientation preserving homeomorphism of the two annuli give by the identification between $H_1$ and $H_2$. Let $k_i:C_{i,1}\times I \rightarrow C_{i,2}\times I$ be an orientation preserving homeomorphism which also preserves the orientation of the $I$-factor. For each $i\in \{1,\cdots,k\}$, we take an orientation reversing homeomorphism $\phi_i:C_{i,1}\times I \cup C_{i,2}\times I\rightarrow A_{i,1} \cup A_{i,2}$ which maps $C_{i,1}\times I$ to $A_{i,1}$ and maps $C_{i,2}\times I$ to $A_{i,2}$, such that $$j_i\circ\phi_i|_{C_{i,1}\times I}=\phi_i|_{C_{i,2}\times I}\circ k_i.$$

Now let $$K=(H_1\cup H_2)\bigcup_{\{\phi_i\}_{i=1}^{k}}(\bigcup_{i=1}^{k} S_i\times I).$$ Then by the construction of $i_0:\widetilde{Z}\rightarrow \mathbb{H}^3$ in Construction \ref{standard}, it is easy to see that $\mathbb{H}^3/\rho_0(\pi_1(Z))$ is homeomorphic with $int(L)$, with respect to their orientations. Actually, each boundary component of $K$ is incompressible.

Let $Z^{(1)}$ be the union of $0$-cells and $1$-cells in $Z$, $\widetilde{Z}^{(1)}$ be the preimage of $Z^{(1)}$ in $\widetilde{Z}$, and $N(\widetilde{Z}^{(1)})$ be a small neighborhood of $i_0(\widetilde{Z}^{(1)})$ in $\mathbb{H}^3$. Then an important point is, $N(\widetilde{Z}^{(1)})/\rho_0(\pi_1(Z))$ is homeomorphic to two copies of $N^{(1)}$, and $(\partial N(\widetilde{Z}^{(1)}))\cap (i_0(\widetilde{Z}))/\rho_0(\pi_1(Z))$ correspond with the disjoint union of circles on the two copies of $\partial N^{(1)}$, where the $2$-handles of $N$ are attached. This observation holds by the geometry of the regular icosahedron, and may not hold if we start with an arbitrary handle structure of $N$.

Here is a simple lemma on constructing non-zero degree maps.

\begin{lem}\label{observation}
If $M'$ is a closed oriented $3$-manifold which contains a codimension-$0$ submanifold $K$, such that there is a proper degree-$d$ map $h:(K,\partial K)\rightarrow (N^{(2)},\partial N^{(2)})$. Then there is a degree-$d$ map $f:M'\rightarrow N$, which is an extension of $h$.
\end{lem}

\begin{proof}
Since each component of $N\setminus N^{(2)}$ is a $3$-handle, it has the $2$-sphere as its boundary. So for each component $Q$ of $M'\setminus int(K)$, $h|_{\partial Q}$ maps $\partial Q$ to the disjoint union of a few $2$-spheres. Suppose $Q$ has $q$ boundary components. Let $[q]$ be the  topological space consists of $q$ points with discrete topology, and let $\cup_{i=1}^q B_i^3$ be the disjoint union of $q$ $3$-balls. Then we use $E_q$ to denote the mapping cone of $[q]\rightarrow \cup_{i=1}^q B_i^3$ which maps the $i$th point $i\in [q]$ to the center of $B_i^3$.

Now we can define a map $e: Q\rightarrow E_q$. On a collar neighborhood of $\partial Q$, $e|_{\partial Q\times I}$ is defined by $$\partial Q\times I\xrightarrow{h|_{\partial Q}\times I} \cup_{i=1}^q S^2\times I\rightarrow \cup_{i=1}^q B^3.$$ For the remaining part of $Q$, $e$ maps it to the graph in $E_q$, which is the cone over $q$ points. Then there is an obvious map from $E_q$ to $N$, which maps the $i$th $3$-ball in $E_q$ to the corresponding $3$-handle in $N$ (whose boundary $2$-sphere is the image of the $i$th boundary component of $Q$) by homeomorphism, and maps the $q+1$-vertex graph in $E_q$ to a graph in $N$ which is the union of $q$ paths.

The composition of these two maps gives $f|_Q:Q\rightarrow N$, and $f|_K$ is defined to be equal to $h$. So we get a map $f:(M',K)\rightarrow (N,N^{(2)})$ such that $f|_K: K\rightarrow N^{(2)}$ is a proper degree-$d$ map and $f(M'\setminus K)\cap N^{(2)}$ is a $1$-dimensional subcomplex in $N^{(2)}$. So for any generic point $m\in int(N^{(2)})$, $f^{-1}(m)=h^{-1}(m)$, and the sums of local degrees are both equal to $\deg(h)=d$. So $f:M'\rightarrow N$ is a degree-$d$ map.
\end{proof}

Now we are ready to prove Theorem \ref{main}:

\begin{proof}

For any closed hyperbolic $3$-manifold $M$, we have already constructed an immersed $\pi_1$-injective $2$-complex $j:Z\looparrowright M$, such that $\mathbb{H}^3/j_*(\pi_1(Z))$ is homeomorphic to $\mathbb{H}^3/\rho_0(\pi_1(Z))$, which is homeomorphic to $int(K)$, with respect to their orientations. Let $q:\bar{M}\rightarrow M$ be the infinite cover of $M$ such that $q_*(\pi_1(\bar{M}))=j_*(\pi_1(Z))$, then $\bar{M}$ is homeomorphic to $int(K)$ with respect to their orientations. By shrinking the boundary of $K$ into its interior a little bit, we get an embedding $\iota:K\hookrightarrow \bar{M}$, and $\iota(K)$ is a compact subset of $\bar{M}$.

Building on Wise's work (\cite{Wi}), Agol showed that the groups of hyperbolic $3$-manifolds are LERF (\cite{Ag}). By Scott's equivalent formulation of LERF (\cite{Sc}), there is an intermediate finite cover $M'\rightarrow M$ of $\bar{M}\rightarrow M$, such that $\iota(K)\subset \bar{M}$ projects into $M'$ by homeomorphism. We still use $K$ to denote the projection of $\iota(K)$ in $M'$, since it is homeomorphic to $K$.

Now we need only to construct a proper degree-$2$ map $h:(K,\partial K)\rightarrow (N^{(2)},\partial N^{(2)})$, then Lemma \ref{observation} implies $M'$ $2$-dominates $N$.

$h|_{H_1\cup H_2}$ maps $H_1$ and $H_2$ to $N^{(1)}$ by the identification between $H$ and $N^{(1)}$, which is an orientation preserving trivial $2$-sheet cover. Moreover, $h$ maps $A_{i,1}$ and $A_{i,2}$ to $A_i$ by orientation-preserving homeomorphism.

For $K\setminus int(H_1\cup H_2)=\cup (S_i\times I)$, each $S_i\times I$ is mapped to the $i$th $2$-handle of $N^{(2)}$. Recall that $S_i \times I$ is pasted to $H_1 \cup H_2$ by an orientation reversing homeomorphism $\phi_i:C_{i,1}\times I \cup C_{i,2}\times I\rightarrow A_{i,1} \cup A_{i,2}$, and the $i$th $2$-handle $D_i\times I$ of $N$ is pasted to $N^{(1)}$ by an orientation reversing homeomorphism $(\partial D_i)\times I\rightarrow A_i$. Here we can assume the orientation of the image of $\partial D_i$ and the orientation of the image of $C_{i,1}$ and $C_{i,2}$ on $\partial N^{(1)}$ coincide. Then there is a proper degree-$2$ map $p_i:S_i\rightarrow D_i$ such that $p_i|_{C_{i,1}}: C_{i,1}\rightarrow \partial D_i$ and $p_i|_{C_{i,2}}: C_{i,2}\rightarrow \partial D_i$ are orientation preserving homeomorphisms. Here we can first pinch $S_i$ to the wedge of two discs, then an obvious map to the disc can be defined. Now we define $h|_{S_i\times I}:S_i\times I \rightarrow D_i\times I$ by $p_i\times I$, which extends the definition of $h$ on $H_1\cup H_2$ to $K$. Then $h:(K,\partial K) \rightarrow (N^{(2)},\partial N^{(2)})$ is a proper degree-$2$ map, by considering any point in $int(H_1\cup H_2)$.
\end{proof}

\section{$\pi_1$-Injectivily of Immersed $2$-complex}\label{injectivity}

In this section, we will prove two technical results which are stated in Section \ref{construction2} without a proof: Proposition \ref{quasi} and Theorem \ref{claim}. These two results are both proved by estimation in hyperbolic geometry. To prove the first result, we need only to do estimation on the standard model (union of totally geodesic subsurfaces in $\mathbb{H}^3$). To prove the second result, we need to do estimation on small deformations of the standard model, which is more complicated. Luckily, we have established some estimations in \cite{Su}, which will be useful for our proof.

Actually, Proposition \ref{quasi} and Theorem \ref{claim} also hold for other nice-constructed immersed $2$-complexes in closed hyperbolic $3$-manifolds, by using Kahn-Markovic and Liu-Markovic's immersed almost totally geodesic subsurfaces. However, it is difficult to give a clean formulation for the general case, so we only deal with the special case which is necessary for this paper.

\subsection{Modified Path and Estimation for the Standard Model}

In this subsection, we will first introduce the concept of {\it modified path}, then prove Proposition \ref{quasi}.

Let $Z^{(0)}$ be the union of $0$-cells in $Z$, and let $Z^{(1)}$ be the union of $0$- and $1$-cells in $Z$. We use $\widetilde{Z}^{(0)}$ and $\widetilde{Z}^{(1)}$ to denote the preimage of $Z^{(0)}$ and $Z^{(1)}$ in $\widetilde{Z}$ respectively. The closure of a component of $\widetilde{Z}\setminus \widetilde{Z}^{(1)}$ will be called a {\it piece} of $\widetilde{Z}$.

For a path in $\widetilde{Z}$ which lies in a piece of $\widetilde{Z}$, if $i_0$ maps this path to a geodesic arc in $\mathbb{H}^3$, we simply call it a geodesic arc in $\widetilde{Z}$. For two geodesic arcs in $\widetilde{Z}$ sharing a common point, we measure the angle between these two geodesic arcs by measuring the angle between their images in $\mathbb{H}^3$.

For any two points $x,y \in \widetilde{Z}$, let $\gamma$ be the shortest path in $(\widetilde{Z},d_0)$ from $x$ to $y$. Then $\gamma$ is a concatenation of geodesic arcs, and each geodesic arc lies in a piece of $\widetilde{Z}$. In the concatenation, $\gamma$ may contain some relatively short geodesic arcs, which is not convenient for our geometric estimation. So we first introduce the concept of {\it modified path}, which eliminates such short geodesic arcs.

For the oriented shortest path $\gamma$ in $(\widetilde{Z},d_0)$ from $x$ to $y$, if $\gamma$ does not contain any geodesic arc in $\widetilde{Z}^{(1)}$, let $x_1,x_2,\cdots,x_n$ be the sequence of intersection points in $\widetilde{Z}^{(1)}\cap (\gamma\setminus \{x,y\})$, which follows the orientation of $\gamma$. If $\gamma$ contains some geodesic arc in $\widetilde{Z}^{(1)}$, for each geodesic arc in $\widetilde{Z}^{(1)} \cap \gamma$, we only record its initial and terminal points in $\widetilde{Z}^{(0)}$. In this case, let $x_1,x_2,\cdots,x_n$ be the sequence of transversal intersection points in $\widetilde{Z}^{(1)}\cap (\gamma\setminus \{x,y\})$ and the initial and terminal points of geodesic arcs in $\widetilde{Z}^{(1)} \cap \gamma$, such that this sequence is following the orientation of $\gamma$. The sequence $x_1,x_2,\cdots,x_n$ is called the {\it intersection sequence} of $\gamma$.

For the intersection sequence $x_1,x_2,\cdots,x_n$, we do the following process to define the {\it modified sequence}, by beginning with $x_1$. If $d_0(x_1,x_2)\geq \frac{R}{64}$, then we put $x_1$ at the first position of the modified sequence, and denote it by $y_1$. If $d_0(x_1,x_2)<\frac{R}{64}$, let $x_1,x_2,\cdots,x_j$ be the maximal consecutive subsequence of $x_1,x_2,\cdots,x_n$, such that $d_0(x_i,x_{i+1})<\frac{R}{64}$ holds for $i=1,2,\cdots,j-1$. Then the $1$-cells containing $x_1,x_2,\cdots,x_j$ share a common $0$-cell, and denote it by $y_1$. We put $y_1$ at the first position of the modified sequence, and call it a {\it modified point} in this case. Then we turn to consider $x_2$ in the first case ($d_0(x_1,x_2)\geq \frac{R}{64}$), and consider $x_{j+1}$ in the second case ($d_0(x_1,x_2)<\frac{R}{64}$). By applying the same process inductively, we get the modified sequence $y_1,y_2,\cdots,y_m$.

Note that the maximal consecutive subsequence $x_i,\cdots,x_{i+j}$ satisfying $d_0(x_{i+k},x_{i+k+1})$ for $k=0,\cdots,j-1$ has length at most $3$, since the angle between two $1$-cells sharing a $0$-cell in $\widetilde{Z}$ is $2\theta_0\approx 0.3524\pi$, and $3\cdot0.3524\pi>\pi$.

By the definition of modified sequence, two adjacent points $y_i$ and $y_{i+1}$ in the modified sequence lie in the same piece of $\widetilde{Z}$. Then the {\it modified path} $\gamma'$ of $\gamma$ is defined to be the concatenation of geodesic arcs connecting $x$ to $y_1$, $y_1$ to $y_2$, $\cdots$, $y_m$ to $y$, and each such geodesic arc lies in a piece of $\widetilde{Z}$.

Now we are ready to prove Proposition \ref{quasi}, and we restate the result here. In the following part of this section, we will use $d$ to denote the metric $d_{\mathbb{H}^3}$ on $\mathbb{H}^3$, when it does not cause any confusion.

\begin{prop}\label{quasi1}
When $R$ is large enough, $i_0:(\widetilde{Z},d_0) \rightarrow (\mathbb{H}^3, d_{\mathbb{H}^3})$ is an embedding, and also a quasi-isometric embedding. In particular, $\rho_0:\pi_1(Z)\rightarrow PSL_2(\mathbb{C})$ is an injective map.
\end{prop}

\begin{proof}
For any $x,y \in \widetilde{Z}$, with $d_0(x,y)\geq\frac{R}{4}$, let $\gamma$ be the shortest path in $(\widetilde{Z},d_0)$ from $x$ to $y$, with intersection sequence $x_1,x_2,\cdots,x_n$. Let $\gamma'$ be the modified path of $\gamma$, and let $y_1,y_2,\cdots,y_m$ be the modified sequence.

For any two adjacent points $y_i$ and $y_{i+1}$ in the modified sequence, if the geodesic arc in $\widetilde{Z}$ from $y_i$ to $y_{i+1}$ intersects with the preimage of some $S_i$, then $d_0(y_i,y_{i+1})\geq \frac{R}{2}$, by step IV in the construction of $Z$ (Section \ref{construction2}).

Now we suppose the geodesic arc from $y_i$ to $y_{i+1}$ does not intersect with the preimage of any $S_i$. If both $y_i$ and $y_{i+1}$ are modified points, then they are two distinct $0$-cells in $\widetilde{Z}$, so $d_0(y_i,y_{i+1})\geq R$. If both of them are not modified points, we have $d_0(y_i,y_{i+1})\geq \frac{R}{64}$. If one of them is a modified point, and the other one is not, then $d_0(y_i,y_{i+1})\geq \frac{R}{128}$ holds, by hyperbolic geometry on $\mathbb{H}^2$ and $2\theta_0\approx 0.3524\pi$. So $d_0(y_i,y_{i+1})\geq \frac{R}{128}$ always holds.

Now we claim that $\angle{y_{i-1}y_iy_{i+1}}$ in $\mathbb{H}^3$ has some positive lower bound, say $\angle{y_{i-1}y_iy_{i+1}}>\frac{\pi}{18}$ in $\mathbb{H}^3$. Here "$\angle{y_{i-1}y_iy_{i+1}}$ in $\mathbb{H}^3$" actually means $\angle{i_0(y_{i-1})i_0(y_i)i_0(y_{i+1})}$, but we use this notation to simplify the expression.

Proof of the claim:\\
{\bf Case I.} $y_i$ is not a modified point, so $y_i=x_j$. Since $\angle{x_{j-1}x_jx_{j+1}}\geq \frac{2\pi}{5}$ in $\mathbb{H}^3$, we need only to show that $\angle{x_{j+1}y_iy_{i+1}}$ and $\angle{x_{j-1}y_iy_{i-1}}$ are both very small. Here we only show that $\angle{x_{j+1}y_iy_{i+1}}$ is very small, and the proof for $\angle{x_{j-1}y_iy_{i-1}}$ is exactly the same.

If $y_{i+1}$ is not a modified point, then $x_{j+1}=y_{i+1}$, so $\angle{x_{j+1}y_iy_{i+1}}=0$.

So we suppose $y_{i+1}$ is a modified point, and $y_{i+1}$ corresponds with $x_{j+1}$. If the shortest path from $y_i$ to $y_{i+1}$ intersect with the preimage of some $S_i$, then $d_0(y_i,y_{i+1})\geq\frac{R}{2}$. By hyperbolic geometry, we know that $d_0(x_{j+1},y_{i+1})\leq \frac{R}{64}+1$. So $\angle{x_{j+1}y_iy_{i+1}}\leq e^{-\frac{R}{4}}\leq 8e^{-\frac{R}{64}}$ in $\mathbb{H}^3$.

If the shortest path from $y_i$ to $y_{i+1}$ does not intersect with the preimage of any $S_i$, then the flattened picture (put two adjacent pieces of $\widetilde{Z}$ in the same hyperbolic plane) is shown in Figure 6 a). We only draw an Euclidean picture to show the position of these points and geodesics, although the real picture lies in $\mathbb{H}^2$. The following computation in the proof of Case I will be in $\mathbb{H}^2$.

By Construction \ref{standard}, we have $\angle{y_iy_{i+1}x_{j+1}}\leq\angle{x_{j+1}y_{i+1}x_{j+2}}=2\theta_0$. Since $y_i$ is not a modified point and $y_{i+1}$ is, $d(y_i,x_{j+1})\geq\frac{R}{64}$ and $d(x_{j+1},x_{j+2})\leq\frac{R}{64}$. Then by hyperbolic geometry in $\mathbb{H}^2$, we have
\begin{equation}\label{3}
\frac{\sinh{d(x_{j+1},x_{j+2})}}{\sin{2\theta_0}}\geq \sinh{d(y_{i+1},x_{j+2})}
\end{equation}
and
\begin{equation}\label{4}
\frac{\sinh{d(y_i,x_{j+2})}}{\sin{\angle{y_iy_{i+1}x_{j+2}}}}=\frac{\sinh{d(y_{i+1},x_{j+2})}}{\sin{\angle{x_{j+1}y_iy_{i+1}}}}.
\end{equation}
Here $2\theta_0\leq \angle{y_iy_{i+1}x_{j+2}} \leq 4\theta_0$.

So
\begin{equation}\label{5}
\begin{split}
& \sin{\angle{x_{j+1}y_iy_{i+1}}}=\frac{\sinh{d(y_{i+1},x_{j+2})}\cdot\sin{\angle{y_iy_{i+1}x_{j+2}}}}{\sinh{d(y_i,x_{j+2})}} \leq\frac{\sinh{d(y_{i+1},x_{j+2})}}{2\sinh{d(y_i,x_{j+1})}\cdot \sinh{d(x_{j+1},x_{j+2})}}\\
& \leq \frac{\sinh{d(x_{j+1},x_{j+2})}}{2\sinh{d(y_i,x_{j+1})}\cdot \sinh{d(x_{j+1},x_{j+2})}\cdot\sin{2\theta_0}}=\frac{1}{\sinh{d(y_i,x_{j+1})}\cdot 2\sin{2\theta_0}}\leq 4e^{-\frac{R}{64}}.
\end{split}
\end{equation}
In this case $\angle{x_{j+1}y_iy_{i+1}}\leq 8e^{-\frac{R}{64}}$ in $\mathbb{H}^3$ also holds, and the same estimation holds for $\angle{x_{j-1}y_iy_{i-1}}$.

Since the angle between two adjacent pieces of $\widetilde{Z}$ in $\mathbb{H}^3$ is at least $\frac{2\pi}{5}$, and $\gamma$ is the shortest path in $(\widetilde{Z},d_0)$, $\angle{x_{j-1}x_jx_{j+1}}\geq \frac{2\pi}{5}$ holds in $\mathbb{H}^3$. So by the above estimation, when $y_i$ is not a modified point, $\angle{y_{i-1}y_iy_{i+1}}\geq \frac{\pi}{5}$ in $\mathbb{H}^3$.

{\bf Case II.} $y_i$ is a modified point, then $y_i$ corresponds with a consecutive subsequence $x_j,x_{j+1}$ or $x_j,x_{j+1},x_{j+2}$ of $x_1,x_2,\cdots,x_n$.

We first suppose that both $y_{i-1}$ and $y_{i+1}$ are not modified points, then the flattened picture for the first subcase ($y_i$ corresponds with $x_j,x_{j+1}$) is shown in Figure 6 b). So in this flattened picture, $\angle{x_{j-1}y_ix_{j+2}}\leq \pi$. Since a path starting at some point in $\widetilde{Z}$ need to go through at least two other pieces of $\widetilde{Z}$ to return to the original piece and $2\theta_0\approx 0.3524\pi$, the geometry gives us $\angle{y_{i-1}y_iy_{i+1}}=\angle{x_{j-1}y_ix_{j+2}}\geq 3\times 0.3524\pi-\pi= 0.0572\pi>\frac{\pi}{18}$ in $\mathbb{H}^3$. The proof for the second subcase ($y_i$ corresponds with $x_j,x_{j+1},x_{j+2}$) is exactly the same.

\begin{center}
\psfrag{a}[]{$y_i=x_j$} \psfrag{b}[]{$x_{j+1}$} \psfrag{c}[]{$x_{j+2}$}
\psfrag{d}[]{$y_{i+1}$} \psfrag{e}[]{a)} \psfrag{f}[]{b)} \psfrag{g}[]{$x_{j-1}=y_{i-1}$}
\psfrag{h}[]{$y_i$} \psfrag{i}[]{$x_{j+2}=y_{i+1}$} \psfrag{j}[]{$x_{j}$} \psfrag{k}[]{$x_{j+1}$}
\includegraphics[width=5.5in]{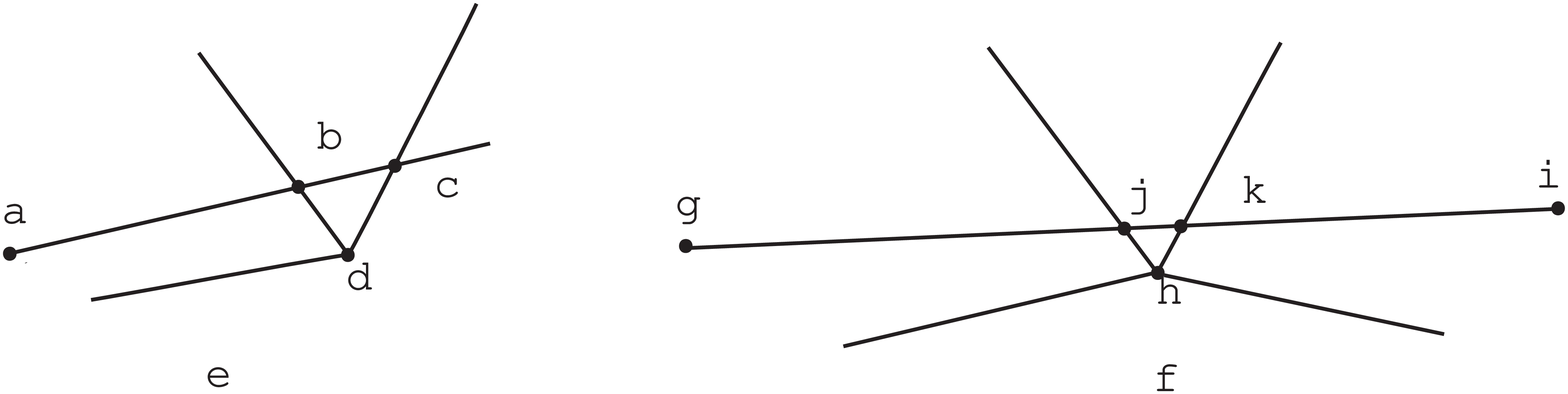}
 \centerline{Figure 6}
\end{center}

If $y_{i-1}$ is a modified point, then $d_0(y_{i-1},x_{j-1})\leq \frac{R}{64}+1$. Moreover, since $y_{i-1}$ and $y_i$ are two different points in $\tilde{Z}^{(0)}$, $d_0(y_{i-1},y_i)\geq R$ holds, so $\angle{x_{j-1}y_iy_{i-1}}\leq e^{-\frac{R}{2}}$. The same argument shows that $\angle{x_{j+2}y_iy_{i+1}}\leq e^{-\frac{R}{2}}$ holds when $y_{i+1}$ is a modified point. So in this case $\angle{y_{i-1}y_iy_{i+1}}\geq 0.0572\pi -2e^{-\frac{R}{2}}> \frac{\pi}{18}$ in $\mathbb{H}^3$, and the proof of the claim is done.

So we know that $d_0(y_i,y_{i+1})\geq \frac{R}{128}$ and $\angle{y_{i-1}y_iy_{i+1}}\geq \frac{\pi}{18}$ in $\mathbb{H}^3$. Actually, the same argument implies $\angle{xy_1y_2}\geq \frac{\pi}{18}$ in $\mathbb{H}^3$ when $d_0(x,x_1)\geq \frac{R}{32}$, and so does $\angle{y_{m-1}y_my}$.

Since $d_0$ is the path metric induced by the hyperbolic metric $d_{\mathbb{H}^3}$,
\begin{equation}\label{6}
d_{\mathbb{H}^3}(i_0(x),i_0(y))\leq d_0(x,y).
\end{equation}
On the other hand, since $\gamma'$ is a path from $x$ to $y$,
\begin{equation}\label{7}
d_0(x,y)\leq d_0(x,y_1)+d_0(y_m,y)+\sum_{j=1}^{m-1} d_0(y_j,y_{j+1}).
\end{equation}

Lemma 4.8 (1) of \cite{LM} implies that, when $d_0(x,x_1), d_0(y,x_n) \geq \frac{R}{32}$,
\begin{equation}\label{8}
\begin{split}
& d_{\mathbb{H}^3}(i_0(x),i_0(y)) \geq d_0(x,y_1)+d_0(y_m,y)+\sum_{j=1}^{m-1} d_0(y_j,y_{j+1})-m(2\log{(\sec{\frac{17\pi}{36}})}+1)\\
& \geq \frac{9}{10}(d_0(x,y_1)+d_0(y_m,y)+\sum_{j=1}^{m-1} d_0(y_j,y_{j+1}))\geq \frac{9}{10}d_0(x,y).
\end{split}
\end{equation}

If $d_0(x,x_1)$ or $d_0(y,x_n)$ is less than $\frac{R}{32}$, the same argument gives
\begin{equation}\label{9}
d_{\mathbb{H}^3}(i_0(x),i_0(y))\geq \frac{9}{10}d_0(x,y)-\frac{R}{8}.
\end{equation}

Now we have that $i_0:(\widetilde{Z},d_0) \rightarrow (\mathbb{H}^3, d_{\mathbb{H}^3})$ is a quasi-isometric embedding, so $\rho_0:\pi_1(Z)\rightarrow PSL_2(\mathbb{C})$ is an injective map. Actually, the above argument shows that $d(i_0(x),i_0(y))\geq\frac{R}{10}$ if $d_0(x,y)\geq \frac{R}{4}$. Moreover, $i_0(x)\ne i_0(y)$ holds if $0<d_0(x,y)<\frac{R}{4}$, by the local geometry of $i_0(\widetilde{Z})$. So we have that $i_0:\widetilde{Z} \rightarrow \mathbb{H}^3$ is an embedding.
\end{proof}

\subsection{Estimation of the Deformations of $i_0$}

Now we turn to prove Theorem \ref{claim}. Actually, $j_*:\pi_1(Z)\rightarrow \pi_1(M)\subset PSL_2(\mathbb{C})$ lies in a continuous family of small deformations of $\rho_0:\pi_1(Z)\rightarrow PSL_2(\mathbb{C})$, and we will show that all the representations in this family satisfy Theorem \ref{claim}. Each such representation $\rho:\pi_1(Z)\rightarrow PSL_2(\mathbb{C})$ is accompanied with a $\pi_1(Z)$-equivariant partially defined map $i:\widetilde{Z}\rightarrow \mathbb{H}^3$, which serves as the geometric realization of $\rho$. We will do geometric estimation for $i:\widetilde{Z}\rightarrow \mathbb{H}^3$ to deduce desired properties of $\rho$.

Now we define an $1$-dimensional subcomplex $W\subset Z$. The $0$-cells, $1$-cells of $Z$, and all the circles in the pants decomposition of $S_i$ ($i=1,2,\cdots,k$) are contained in $W$. Besides these parts, for each pair of pants in $S_i$, three seams of the pants are contained in $W$; for each cornered annulus in $Z$, a few (one, two or four, depends on the number of corners of this cornered annulus) disjoint arcs from the corners to the circle boundary component are also contained in $W$. Moreover, the position of the end points  of these arcs on the circles are given by the following assignment. When considering $i_0:\widetilde{Z}\rightarrow \mathbb{H}^3$, each of the preimage of these arcs is mapped to a geodesic arc which is perpendicular with the corresponding bi-infinite geodesics it intersects with.

By the definition of $W$, each component of $Z\setminus W$ is a topological disc. Let $\widetilde{W}$ be the preimage of $W$ in $\widetilde{Z}$. Then for the embedding $i_0|_{\widetilde{W}}:\widetilde{W}\rightarrow \mathbb{H}^3$, all pair of geodesic (arcs) in $i_0(\widetilde{W})$ that intersect with each other are perpendicular with each other, except the case that the intersection point lies in $i_0(\widetilde{Z}^{(0)})$.

Now let us define a map $i:\widetilde{W}\rightarrow \mathbb{H}^3$, which realizes some $\rho:\pi_1(Z)\rightarrow PSL_2(\mathbb{C})$. Before the construction of $i$, we first need to choose some parameters. Recall that, for each $S_i$, $j(\partial S_i)$ is a multicurve in $\mathbb{Z}{\bold \Gamma}_{r_i R,\frac{\epsilon}{r_i R}}$ for $r_i \in \{1,2,4\}$.

\begin{parameter}\label{parameter}
\begin{itemize}
\item For each $1$-cell $t$ in $Z$, it is associate with the following parameters.
\begin{itemize}
\item A real number $\delta_t$, such that $|\delta_t|< \frac{\epsilon}{R}$.
\item Two vectors $\vec{u}_{1,t},\vec{u}_{2,t}\in S^2$, such that the angles between $\vec{u}_{1,t},\vec{u}_{2,t}$ and the corresponding vectors $\vec{v}_1, \vec{v}_2$ in Table 1 are less than $\frac{\epsilon}{R}$ respectively.
\item An element $S_t\in SO(3)$, such that $\Theta(\vec{v},S_t\vec{v})<\frac{\epsilon}{R}$ for any $\vec{v}\in S^2$.
\end{itemize}
\item  For each circle $C$ in the pants decomposition of $S_i$ which does not lie on the boundary of $S_i$, it is associated with two complex numbers $\xi_C$ and $\eta_C$, such that $|\xi_C|< \frac{\epsilon}{r_i R}$ and $|\eta_C|< \frac{\epsilon}{(r_i R)^2}$.
\item For each boundary component $C$ of $S_i$, it is associated with a complex number $\eta_C$ such that $|\eta_C|< \frac{\epsilon}{(r_i R)^2}$.
\end{itemize}
\end{parameter}

Then the map $i:\widetilde{W}\rightarrow \mathbb{H}^3$ is defined by the following conditions. Note that the parameters are given for $1$-cells and circles in $Z$, but not in $\widetilde{Z}$. So $i:\widetilde{W}\rightarrow \mathbb{H}^3$ is $\pi_1(Z)$-equivariant for some $\rho: \pi_1(Z)\rightarrow PSL_2(\mathbb{C})$, and the following conditions also define this $\rho:\pi_1(Z)\rightarrow PSL_2(\mathbb{C})$.

\begin{construction}\label{defor}
\begin{itemize}
\item Each bi-infinite line in $\widetilde{W}$ is subdivided to concatenation of compact arcs by its intersection with other arcs in $\widetilde{W}$. Then all the arcs in $\widetilde{W}$ are mapped to geodesic arcs in $\mathbb{H}^3$ (under the arc length parametrization induced by $(\widetilde{Z},d_0)$).
\item For each $0$-cell $x \in \widetilde{W}$ which is also a $0$-cell in $\widetilde{Z}$, it is associated with an orthonormal frame $\{\vec{e}_1,\vec{e}_2,\vec{e}_3\}$ at $i(x)$ with respect to the orientation of $\mathbb{H}^3$.
\item For each $1$-cell $t' \in \widetilde{W}$ which is also a $1$-cell in $\widetilde{Z}$, it is endowed with an orientation such that the oriented $1$-cell $t=p(t')$ in $Z$ corresponds with one of the oriented $1$-cell $a,b,c,d,e,f$ in $X'$. Suppose $t'$ travels from one $0$-cell $x$ to another $0$-cell $y$ in $\widetilde{Z}$, let $\{\vec{e}_1,\vec{e}_2,\vec{e}_3\}$ and $\{\vec{e}_1',\vec{e}_2',\vec{e}_3'\}$ be the two orthonormal frames at $i(x)$ and $i(y)$ respectively, then the following conditions hold.
\begin{itemize}
\item $i(t')$ is a geodesic arc from $i(x)$ to $i(y)$ with length equals $L+\delta_t$.
\item The tangent vectors of $i(t')$ at $i(x)$ and $i(y)$ are equal to $\vec{u}_{1,t}$ and $\vec{u}_{2,t}$ respectively (under the coordinate given by frames $\{\vec{e}_1,\vec{e}_2,\vec{e}_3\}$ and $\{\vec{e}_1',\vec{e}_2',\vec{e}_3'\}$).
\item The parallel transportation of $\{\vec{e}_1,\vec{e}_2,\vec{e}_3\}$ to $i(y)$ along $i(t')$ is equal to the composition of $S_t$ and the counterclockwise $\pi-2\theta_0$ rotation of $\{\vec{e}_1',\vec{e}_2',\vec{e}_3'\}$ about the vector $\vec{n} \in T_{i(q)}^1(M)$ in \eqref{2}.
\end{itemize}
\item For any circle $C$ in the pants decomposition of some $S_i$, the image of each component of $p^{-1}(C)$ under $i$ is a bi-infinite geodesic in $\mathbb{H}^3$. For any circle $C$ which lies in $\partial S_i$, the image of each component of $p^{-1}(C)$ under $i$ share the same limit points with the corresponding concatenation of geodesic arcs, on $\partial \mathbb{H}^3=S^2_{\infty}$.
\item For any arc in $\widetilde{W}$ which corresponds with a seam in a pair of pants, or goes from a corner point to the opposite circle in a cornered annulus of $Z$, it is mapped to a geodesic arc which is perpendicular with the corresponding bi-infinite geodesic.
\item For any circle $C$ in $S_i$ shared by two pair of pants $\Pi_1$ and $\Pi_2$ in the oriented surface $S_i$,
$$
\left\{ \begin{array}{l}
        \bold{hl}_{\Pi_1}(C)=\bold{hl}_{\Pi_2}(C)=\frac{r_i R}{2}+\xi_C, \\
        s(C)=1+\eta_C,
\end{array} \right.
$$
holds. For the definition of $\bold{hl}_{\Pi}(C)$ and $s(C)$ in this context, see Construction \ref{standard}.
\item For any oriented boundary component $C_i$ of $S_i$ which is the cuff of a pair of pants $\Pi\subset S_i$, in step III of the construction of $j:Z\rightarrow M$, we have chosen an arc $\alpha_i\subset W$ which goes from some $0$-cell of $Z$ to $C_i$. Take an arbitrary bi-infinite line $\beta\subset \widetilde{W}$ which is a component of $p^{-1}(C_i)$, and any arc $\alpha$ in $\widetilde{W}$ that intersects with $\beta$ and projects to $\alpha_i$. Let $\alpha'$ be the arc that intersects with $\beta$ and projects to a seam of $\Pi$, such that $i_0(\alpha')$ is the nearest such arc from $i_0(\alpha)$. Then the tangent vector of $i(\alpha')$ is the $1+\pi i+\eta_{C_i}$ translation of the tangent vector of $i(\alpha)$ along $i(\beta)$. Moreover, $|\bold{hl}_{\Pi}(C_i)-\frac{r_iR}{2}|<\frac{\epsilon}{R}|$ holds. Here the value of $\bold{hl}_{\Pi}(C_i)$ is determined by the geometry of $i|_{\widetilde{Z}^{(1)}}$.
\end{itemize}
\end{construction}

Note that when $\delta_t=0$, $\vec{u}_{1,t}$, $\vec{u}_{2,t}$ are equal to the corresponding vectors $\vec{v}_1$, $\vec{v}_2$ respectively, and $S_t=id_{SO(3)}$ for any $1$-cell $t$; while $\xi_C=0$ and $\eta_C=0$ for any circle $C$, the map $i:\widetilde{W}\rightarrow \mathbb{H}^3$ is exactly the restriction of our standard model $i_0:\widetilde{Z}\rightarrow \mathbb{H}^3$ on $\widetilde{W}$ (Construction \ref{standard}). By the construction of $j:Z\rightarrow M$ in Section \ref{construction2}, we can choose proper parameters in Parameter \ref{parameter} such that $i:\widetilde{W}\rightarrow \mathbb{H}^3$ equals $\widetilde{j}|_{\widetilde{W}}:\widetilde{W}\rightarrow \widetilde{M}=\mathbb{H}^3$.

Under small parameters in Parameter \ref{parameter}, $i:(\widetilde{W},d_0)\rightarrow (\mathbb{H}^3,d_{\mathbb{H}^3})$ has the following nice property.
\begin{thm}\label{technical}
There exists constants $\hat{\epsilon}>0$ and $\hat{R}>0$, such that for any positive numbers $\epsilon<\hat{\epsilon}$ and $R>\hat{R}$, the following statement holds. For any parameters in Parameter \ref{parameter} with $\epsilon$ and $R$ as above, the map $i:(\widetilde{W},d_0|_{\widetilde{W}})\rightarrow (\mathbb{H}^3,d_{\mathbb{H}^3})$ given by Construction \ref{defor} is a quasi-isometric embedding.
\end{thm}

Theorem \ref{technical} implies Theorem \ref{claim} by the following argument.

Proof of Theorem \ref{claim}: $i:(\widetilde{W},d_0|_{\widetilde{W}})\rightarrow (\mathbb{H}^3,d_{\mathbb{H}^3})$ is a quasi-isometric embedding clearly implies $j_*:\pi_1(Z)\rightarrow \pi_1(M)\subset PSL_2(\mathbb{C})$ is an injective map, when the parameters are properly chosen. We can choose a continuous family of parameters satisfying conditions in Parameter \ref{parameter}, such that the associated family of maps $i_s:\widetilde{W}\rightarrow \mathbb{H}^3$ ($s\in [0,1]$) connects $i_0|_{\widetilde{W}}:\widetilde{W}\rightarrow \mathbb{H}^3$ to $\widetilde{j}|_{\widetilde{W}}:\widetilde{W}\rightarrow \mathbb{H}^3$. So there exists a continuous family of representations $\rho_s:\pi_1(Z)\rightarrow PSL_2(\mathbb{C})$ such that $\rho_s$ lies in $int(AH(\pi_1(Z)))$ for any $s\in[0,1]$, with $\rho_0$ given by the standard model $i_0:\widetilde{Z}\rightarrow \mathbb{H}^3$ and $\rho_1=j_*$. Here $$AH(\pi_1(Z))=\{\rho:\pi_1(Z)\rightarrow PSL_2(\mathbb{C})|\ \rho {\text \ is\ a\ discrete,\ faithful \ representation} \}/\sim.$$
So $\mathbb{H}^3/j_*(\pi_1(Z))$ is homeomorphic with $\mathbb{H}^3/\rho_0(\pi_1(Z))$ with respect to the induced orientations from $\mathbb{H}^3$, which completes the proof of Theorem \ref{claim}.

So it remains to prove Theorem \ref{technical}.

In Theorem 3.8 and Theorem 3.10 of \cite{Su}, we have given estimations for the quasi-isometric constant and angle change on each component of $\widetilde{W}\cap p^{-1}(S_i)$. To state these results, we need the following setting of coordinates.

For any two points $x,y \in \mathbb{H}^3$, we will use $\overline{xy}$ to denote the oriented geodesic arc in $\mathbb{H}^3$ from $x$ to $y$.

Let $V$ be a component of $\widetilde{W}\cap p^{-1}(S_i)$. Take two points $x,y \in V$, such that $d_0(x,y)\geq \frac{R}{4}$, and $x$ lies on some component $p^{-1}(\partial S_i)\cap V$, which is denoted by $\beta$. We will give coordinates for the tangent vectors of $\overline{i_0(x)i_0(y)}$ and $\overline{i(x)i(y)}$ at $i_0(x)$ and $i(x)$ respectively, under some properly chosen frames. More precisely, endow $\beta$ with an arbitrary orientation, let $\alpha\subset V$ be an arc that intersects with $\beta$ and projects to a seam in $S_i$, such that $\alpha$ is the closest such arc from $x$. Let $\vec{e}_1\in T_{i_0(x)}^1(\mathbb{H}^3)$ be the tangent vector of $i_0(\beta)$ at $i_0(x)$, $\vec{e}_2\in T_{i_0(x)}^1(\mathbb{H}^3)$ be the parallel transportation of the tangent vector of $i_0(\alpha)$ to $i_0(x)$ along $i_0(\beta)$, and $\vec{e}_3\in T_{i_0(x)}^1(\mathbb{H}^3)$ such that the orthonormal frame $(\vec{e}_1,\vec{e}_2,\vec{e}_3)$ gives the orientation of $\mathbb{H}^3$. Let $\theta$ be the angle between $\vec{e}_1$ and the tangent vector of $\overline{i_0(x)i_0(y)}$ at $i_0(x)$, and let $\phi$ be the angle between $\vec{e}_3$ and the tangent vector of $\overline{i_0(x)i_0(y)}$ at $i_0(x)$. Then we define $\Theta(i_0(\beta),i_0(\alpha),\overline{i_0(x)i_0(y)})=(\theta,\phi)$. We also define $\theta'$, $\phi'$ and $\Theta(i(\beta),i(\alpha),\overline{i(x)i(y)})=(\theta',\phi')$ by the same way, with $i_0$ replaced by $i$. Note that $\phi=\frac{\pi}{2}$ since $i_0(V)$ lies in a totally geodesic plane in $\mathbb{H}^3$.

Theorem 3.8 and Theorem 3.10 of \cite{Su} give the following statement.

\begin{thm}\label{piece}
For any $0<\delta<1$, there exists constants $\hat{\epsilon}>0$ and $\hat{R}>0$, such that for any positive numbers $\epsilon<\hat{\epsilon}$ and $R>\hat{R}$, the following estimations hold.
\begin{enumerate}
\item $i|_V:(V,d_0|_V) \rightarrow (\mathbb{H}^3,d_{\mathbb{H}^3})$ is an $(1+k\frac{\epsilon}{R}, k(\epsilon+\frac{1}{R})^{\frac{1}{5}})$-quasi-isometric embedding for a universal constant $k$.
\item  $|\Theta(i_0(\beta),i_0(\alpha),\overline{i_0(x)i_0(y)})-\Theta(i(\beta),i(\alpha),\overline{i(x)i(y)})|=|(\theta-\theta',\phi-\phi')|<(\frac{\delta}{300})^2$.
\end{enumerate}
\end{thm}

The estimation we need should give quasi-isometric constant and angle change on pieces of $\widetilde{Z}$. Each piece of $\widetilde{Z}$ is the union of some component of $p^{-1}(S_i)$ and components of the preimage of cornered annuli. To extend the estimation in Theorem \ref{piece} to pieces of $\widetilde{Z}$, we first need to give some estimation on the preimage of cornered annuli.

The following estimation is intuitive and elementary, so we leave it as an exercise for the readers.

\begin{lem}\label{base1}
Let $x_1,x_2,y_1,y_2$ be four points in $\mathbb{H}^2$, and let $Q$ be the union of the four geodesic arcs $\overline{x_1x_2},\overline{y_1y_2},\overline{x_1y_1},\overline{x_2y_2}$ as in Figure 7 a). The geometry of these four points satisfies: $\overline{y_1y_2}$ is perpendicular with both $\overline{x_1y_1}$ and $\overline{x_2y_2}$, $\angle{y_1x_1x_2}=\angle{y_2x_2x_1}=\theta_0$ and $d(x_1,x_2)=L$.

Let $x_1',x_2',y_1',y_2'$ be another four points in $\mathbb{H}^3$, and let $Q'$ be the union of the four geodesic arcs $\overline{x_1'x_2'},\overline{y_1'y_2'},\overline{x_1'y_1'},\overline{x_2'y_2'}$ as in Figure 7 b). The geometry of these four points satisfies: $\overline{y_1'y_2'}$ is perpendicular with both $\overline{x_1'y_1'}$ and $\overline{x_2'y_2'}$, $|d(x_1,y_1)-d(x_1',y_1')|, |d(x_2,y_2)-d(x_2',y_2')|< \frac{10\epsilon}{R}$ and $|d(x_1',x_2')-L|<\frac{10\epsilon}{R}$. Moreover, the angle between the tangent vector of $\overline{y_1'x_1'}$ (at $y_1'$) and the parallel transportation of the tangent vector of $\overline{y_2'x_2'}$ (at $y_2'$) to $y_1'$ along $\overline{y_2'y_1'}$ is less than $\frac{10\epsilon}{R}$.

\begin{center}
\psfrag{a}[]{$x_1$} \psfrag{b}[]{$x_2$} \psfrag{c}[]{$y_1$}
\psfrag{d}[]{$y_2$} \psfrag{e}[]{$x$} \psfrag{f}[]{$p(x)$} \psfrag{g}[]{a)}
\psfrag{h}[]{$x_1'$} \psfrag{i}[]{$x_2'$} \psfrag{j}[]{$y_1'$} \psfrag{k}[]{$y_2'$} \psfrag{l}[]{b)}
\includegraphics[width=5in]{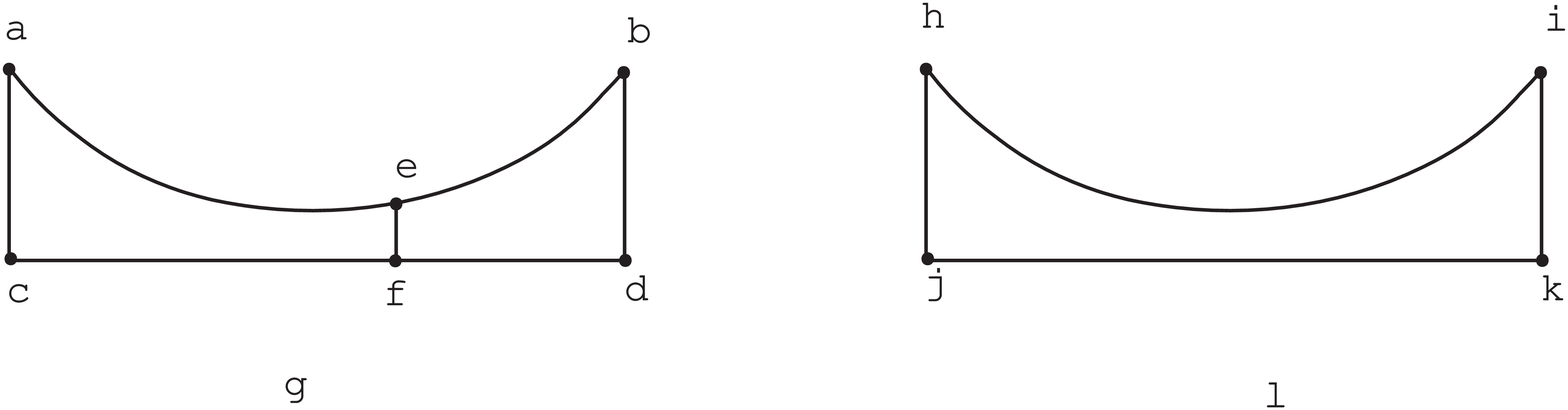}
 \centerline{Figure 7}
\end{center}

Let $p: \overline{x_1x_2}\rightarrow \overline{y_1y_2}$ and $p': \overline{x_1'x_2'}\rightarrow \overline{y_1'y_2'}$ be the nearest point projection, and $r:Q\rightarrow Q'$ be the piecewise linear homeomorphism sending $x_1,x_2,y_1,y_2$ to $x_1',x_2',y_1',y_2'$ respectively.

Then for any $x\in \overline{x_1x_2}$, the following estimations hold.
\begin{itemize}
\item $|d(p(x),x)-d(p'(r(x)),r(x))|<80\sqrt{\frac{\epsilon}{R}}$.
\item $d(r(p(x)),p'(r(x)))<80\sqrt{\frac{\epsilon}{R}}$.
\item $|\angle{x_1xp(x)}-\angle{x_1'r(x)p'(r(x))}|<80\sqrt{\frac{\epsilon}{R}}$.
\end{itemize}
\end{lem}

For three points $p,q,r\in\mathbb{H}^3$ not lying on a geodesic, we use $P_{pqr}$ to denote the hyperbolic plane containing $p,q$ and $r$. For another such hyperbolic plane $P_{p'q'r'}$ intersecting with $P_{pqr}$, we use $\Theta(P_{pqr},P_{p'q'r'})\in[0,\frac{\pi}{2}]$ to denote the angle between these two planes.

In $\widetilde{Z}$, each component of the preimage of a cornered annulus is subdivided into union of $4$-gons by $\widetilde{W}$. Let $R$ be the intersection of such a $4$-gon with $\widetilde{W}$, then $i_0(R)\subset \mathbb{H}^2\subset \mathbb{H}^3$ has exactly the same geometry as $Q$ in Lemma \ref{base1}.

To apply Lemma \ref{base1} to compare the geometry between $i_0(R)$ and $i(R)$, we need the following lemma, which shows that $i(R)\subset \mathbb{H}^3$ satisfies the conditions of $Q'$ in Lemma \ref{base1}, and also gives some further estimations. The proof follows from Construction \ref{defor} and elementary hyperbolic geometry.

\begin{lem}\label{preestimation}
Let $x_1,x_2,y_1,y_2$ be the four vertices of $R\subset \widetilde{W}$, such that the position of these four point are as shown in Figure 7 a), then the following estimation holds.
\begin{itemize}
\item $|d(i(x_1),i(x_2))-L|<\frac{\epsilon}{R}$.
\item $|d(i(x_1),i(y_1))-\cosh^{-1}{(\csc{\theta_0})}|, |d(i(x_2),i(y_2))-\cosh^{-1}{(\csc{\theta_0})}|<\frac{9\epsilon}{R}$.
\item The angle between the tangent vector of $\overline{i(y_1)i(x_1)}$ (at $i(y_1)$) and the parallel transportation of the tangent vector of $\overline{i(y_2)i(x_2)}$ (at $i(y_2)$) to $i(y_1)$ along $\overline{i(y_2)i(y_1)}$ is smaller than $\frac{10\epsilon}{R}$.
\item $|d(i(y_1),i(y_2))-R|<\frac{30\epsilon}{R}$.
\item For any $x\in \overline{i(x_1)i(x_2)}$, let $x'$ be the nearest point projection of $i(x)$ on $\overline{i(y_1)i(y_2)}$, then $\Theta(P_{i(x_1)i(x_2)x'},P_{i(y_1)i(y_2)i(x)}),\Theta(P_{i(x_1)i(y_1)i(y_2)},P_{i(x)i(y_1)i(y_2)})<\frac{10\epsilon}{R}$ holds.
\end{itemize}
\end{lem}
\hfill $\qed$

\begin{rem}\label{base2}
Note that since $\widetilde{W}$ may contains other $1$-cells that intersect with $R$, the definition of $i|_R:R\rightarrow i(R)$ does not exactly coincide with $r:Q\rightarrow Q'$. However, since there are at most two such $1$-cells intersecting with $R$, and they only intersect with $\overline{y_1y_2}$, these two maps are only different on $\overline{y_1y_2}$ by an error of at most $100\frac{\epsilon}{R}$. So the estimations in Lemma \ref{base1} still hold in this case, with $80\sqrt{\frac{\epsilon}{R}}$ replaced by $100\sqrt{\frac{\epsilon}{R}}$.
\end{rem}

Take two points $x,y\in \widetilde{Z}^{(1)}$ lying in the same piece of $\widetilde{Z}$, such that $d_0(x,y)\geq \frac{R}{128}$. Let $U$ be the piece of $\widetilde{Z}$ that contains $x$ and $y$, and let $\gamma$ be the shortest path from $x$ to $y$ in $(\widetilde{Z},d_0)$, then $\gamma$ lies in $U$.

Now we need to give estimations for the length and angle change for $\overline{i(x)i(y)}$ under certain coordinate. To formulate the estimation, we need to give the following setting of coordinate of angles, which is similar with the formulation of Theorem \ref{piece}.

Let $t$ be the $1$-cell in $\widetilde{Z}$ which contains $x$. If $x$ lies in $\widetilde{Z}^{(0)}$, choose any such $t$ which lies in $U$, and there are two choices. We give an orientation on $t$ such that $x$ is closer with the initial point than the terminal point of $t$. Let the other oriented $1$-cell which lies in $U$ and share the initial point with $t$ be $t'$, and denote the intersection point of $t$ and $t'$ by $z$.

At the point $i_0(x)\in \mathbb{H}^3$, let $\vec{e}_1\in T_{i_0(x)}^1(\mathbb{H}^3)$ be the tangent vector of $i_0(t)$ at $i_0(x)$. We use $P_{i_0(t)i_0(t')}$ to denote the hyperbolic plane containing $i_0(t)$ and $i_0(t')$, and let $\vec{e}_3$ be the unit normal vector of $P_{i_0(t)i_0(t')}$ at $i_0(x)$. Then we have an orthonormal frame $(\vec{e}_1,\vec{e}_2,\vec{e}_3)$ at $i_0(x)$, with $\vec{e}_2=\vec{e}_3\times \vec{e}_1$.

Let $\vec{v}\in T_{i_0(x)}^1(\mathbb{H}^3)$ be the tangent vector of $\overline{i_0(x)i_0(y)}$ at $i_0(x)$, $\theta$ be the angle between $\vec{v}$ and $\vec{e}_1$, and $\phi$ be the angle between $\vec{v}$ and $\vec{e}_3$. Then we define $\Theta(i_0(t),i_0(t'),\overline{i_0(x)i_0(y)})=(\theta,\phi)$, and note that $\phi=\frac{\pi}{2}$ here.

$\Theta(i(t),i(t'),\overline{i(x)i(y)})=(\theta',\phi')$ is defined by the same way, with $i_0$ replaced by $i$ in the definition. We also have an orthonormal frame $(\vec{e}_1',\vec{e}_2',\vec{e}_3')$ at $i(x)$ (as frame $(\vec{e}_1,\vec{e}_2,\vec{e}_3)$ at $i_0(x)$ ). Then we have the following estimation.

\begin{prop}\label{estimation}
For any $0<\delta<1$, there exists constants $\hat{\epsilon}>0$ and $\hat{R}>0$, such that for any positive numbers $\epsilon<\hat{\epsilon}$ and $R>\hat{R}$, the following statement holds. For two points $x,y \in \widetilde{Z}^{(1)}$ as above, with corresponding $1$-cells $t$ and $t'$, the following estimations hold.
\begin{itemize}
\item $\frac{1}{2}d(i_0(x),i_0(y))<d(i(x),i(y))<2d(i_0(x),i_0(y))$.
\item $|\Theta(i_0(t),i_0(t'),\overline{i_0(x)i_0(y)})-\Theta(i(t),i(t'),\overline{i(x)i(y)})|=|(\theta-\theta',\frac{\pi}{2}-\phi')|<\delta$.
\end{itemize}
\end{prop}

\begin{proof}
Let $\gamma$ be the shortest path in $(\widetilde{Z},d_0)$ from $x$ to $y$, then $\gamma$ lies in a piece of $\widetilde{Z}$ by assumption, which is denoted by $U$. Let $\widetilde{S}$ be the component of the preimage of $S_i$ which is contained in $U$, and let $\beta$ be the component of $\partial \widetilde{S}$ that is the closest one to $x$.

The estimations clearly holds if $\gamma\subset \widetilde{Z}^{(1)}$. So there are two cases to consider, either $\gamma$ does not intersect with $\widetilde{S}$ (but does not lie in $\tilde{Z}^{(1)}$), or $\gamma$ intersects with $\widetilde{S}$.

In the following, we suppose that $\epsilon>0$ is so small and $R>0$ is so large such that $\frac{\hat{\epsilon}}{\hat{R}}<\frac{\delta^{10}}{10^{40}}$ holds.

{\bf Case I.} $\gamma$ does not intersect with $\widetilde{S}$, then the picture near $\gamma$ looks like Figure 8, with $d_0(x,y)>\frac{R}{128}$.

Let $\omega$ be the concatenation of geodesic arcs in $U$ from $x$ to $y$ as in Figure 8 (part of $\omega$ is drawn by dashed lines). For each geodesic arc in $\omega$, the length of its image under $i_0$ and under $i$ differ by at most $\frac{\epsilon}{R}$, the angle between the image of two adjacent arcs under $i_0$ (which is $2\theta_0$) and under $i$ differ by at most $\frac{2\epsilon}{R}$. Since $d_0(x,y)>\frac{R}{128}$, an exercise in hyperbolic geometry gives the first estimation:
\begin{equation}\label{10}
\frac{1}{2}d(i_0(x),i_0(y))<d(i(x),i(y))<2d(i_0(x),i_0(y)).
\end{equation}
Actually, the constant $2$ can be replaced by $1+\epsilon'$ for some small positive constant $\epsilon'$.

\begin{center}
\psfrag{a}[]{$x$} \psfrag{b}[]{$y$} \psfrag{c}[]{$\gamma$}
\psfrag{d}[]{$\beta$} \psfrag{e}[]{$t$} \psfrag{f}[]{$t'$}
\psfrag{g}[]{$z$} \psfrag{h}[]{$w$}
\includegraphics[width=4.5in]{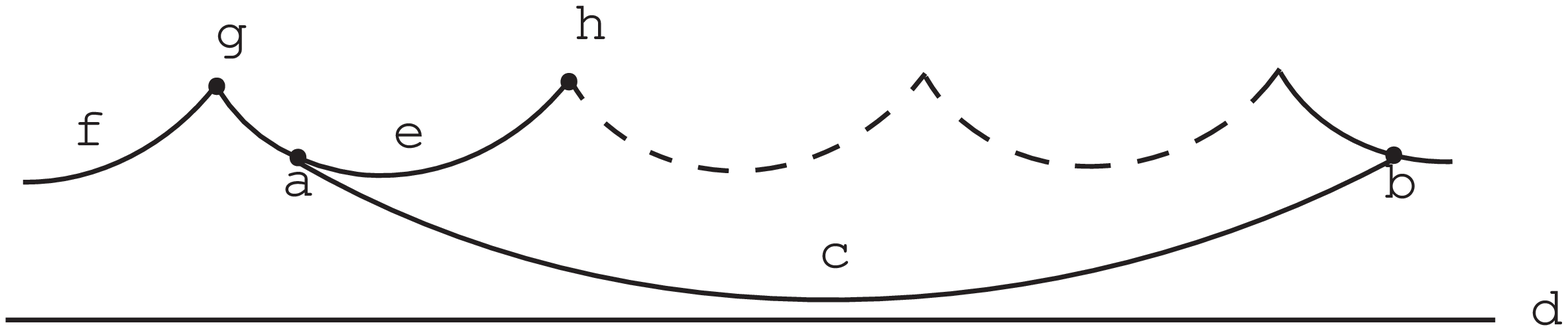}
 \centerline{Figure 8}
\end{center}

Let $x'$ and $y'$ be the nearest point projection of $i_0(x)$ and $i_0(y)$ on $i_0(\beta)$ respectively, and let $x''$ and $y''$ be the nearest point projection of $i(x)$ and $i(y)$ on $i(\beta)$ respectively. Since $\sin{\theta_0}=\sqrt{\frac{5-\sqrt{5}}{10}}$, $d(x',i_0(x)),d(y',i_0(y)),d(x'',i(x)),d(y'',i(y))$ are all less than $\cosh^{-1}{(\csc{\theta_0})}+\frac{10\epsilon}{R}<2$.

By the definition of $\theta$ and $\theta'$, $\theta=\angle{i_0(w)i_0(x)i_0(y)}$ and $\theta'=\angle{i(w)i(x)i(y)}$. Since $d(i_0(x),i_0(y)), d(i(x),i(y))>\frac{R}{256}$ and $d(i_0(y),y')$, $d(i(y),y'')<2$, we have $\angle{i_0(y)i_0(x)y'}$, $\angle{i(y)i(x)y''}<10e^{-\frac{R}{256}}$. Moreover, since $d(i_0(x),y')$, $d(i(x),y'')>\frac{R}{256}-2$, and $|d(i_0(x),x')-d(i(x),x'')|<100\sqrt{\frac{\epsilon}{R}}$ (Lemma \ref{base1}), $|\angle{x'i_0(x)y'}-\angle{x''i(x)y''}|<50(\frac{\epsilon}{R})^{\frac{1}{4}}$ holds.

Note that the fifth estimation in Lemma \ref{preestimation} implies $\Theta(P_{i(z)i(w)x''},P_{i(x)x''y''})<\frac{10\epsilon}{R}$, and the third estimation in Lemma \ref{base1} gives $|\angle{i_0(z)i_0(x)x'}-\angle{i(z)i(x)x''}|<100\sqrt{\frac{\epsilon}{R}}$. So
\begin{equation}\label{11}
\begin{split}
& |\theta-\theta'| =|\angle{i_0(w)i_0(x)i_0(y)}-\angle{i(w)i(x)i(y)}|\\
& \leq |\angle{i_0(w)i_0(x)y'}-\angle{i(w)i(x)y''}|+\angle{i_0(y)i_0(x)y'}+\angle{i(y)i(x)y''}\\
& \leq |\angle{i_0(z)i_0(x)y'}-\angle{i(z)i(x)y''}|+20e^{-\frac{R}{256}}\\
& \leq |\angle{i_0(z)i_0(x)x'}-\angle{i(z)i(x)x''}|+|\angle{x'i_0(x)y'}-\angle{x''i(x)y''}|+\Theta(P_{i(z)i(w)x''},P_{i(x)x''y''})+20e^{-\frac{R}{256}}\\
& \leq 100\sqrt{\frac{\epsilon}{R}}+50(\frac{\epsilon}{R})^{\frac{1}{4}}+\frac{10\epsilon}{R}+20e^{-\frac{R}{256}} < \frac{\delta}{2}.
\end{split}
\end{equation}

We will use $P_{i(t)i(t')}$ to denote the hyperbolic plane containing $i(t)$ and $i(t')$. Let $\vec{v}$ be the tangent vector of $\overline{i(x)i(y)}$ at $i(x)$, and $\vec{n}$ be the normal vector of $P_{i(t)i(t')}$ at $i(x)$, then $\phi'=\Theta(\vec{v},\vec{n})$, so we need to estimate $|\Theta(\vec{v},\vec{n})-\frac{\pi}{2}|$.

Let $\vec{v}'$ be the tangent vector of $\overline{i(x)y''}$ at $i(x)$, we have known that $\Theta(\vec{v}, \vec{v}')=\angle{i(y)i(x)y''}<10e^{-\frac{R}{256}}$. Let $z''$ be the nearest point projection of $i(z)$ on $i(\beta)$, then the fifth estimation of Lemma \ref{preestimation} implies
\begin{equation}\label{12}
\Theta(P_{i(z)z''i(w)},P_{i(z)z''y''}), \Theta(P_{i(z)z''y''},P_{i(x)z''y''})<10\frac{\epsilon}{R}.
\end{equation}
A similar estimation gives
\begin{equation}\label{13}
\Theta{(P_{i(t)i(t')},P_{i(z)z''y''})}<20\frac{\epsilon}{R}.
\end{equation}

So we have
\begin{equation}\label{14}
\begin{split}
& |\Theta(\vec{v},\vec{n})-\frac{\pi}{2}|\leq \Theta(\vec{v}, \vec{v}')+|\Theta(\vec{v}',\vec{n})-\frac{\pi}{2}| \leq 10e^{-\frac{R}{256}}+\Theta(P_{i(t)i(t')},P_{i(x)z''y''})\\
& \leq 10e^{-\frac{R}{256}}+\Theta(P_{i(t)i(t')},P_{i(z)z''i(w)})+\Theta(P_{i(z)z''i(w)},P_{i(x)z''y''})\\
& \leq 10e^{-\frac{R}{256}}+\Theta{(P_{i(t)i(t')},P_{i(z)z''y''})}+2\Theta(P_{i(z)z''i(w)},P_{i(z)z''y''})+\Theta(P_{i(z)z''y''},P_{i(x)z''y''})\\
& \leq 10e^{-\frac{R}{256}}+50\frac{\epsilon}{R}<\frac{\delta}{2}.
\end{split}
\end{equation}

So $|(\theta-\theta',\frac{\pi}{2}-\phi')|<\delta$.

{\bf Case II.} $\gamma$ does intersect with $\widetilde{S}$, then $d_0(x,y)>\frac{R}{2}$ by step III of the construction in Section \ref{construction2}. Let $\bar{x}$ be the intersection point of $\beta$ and $\gamma$, and let $\bar{y}$ be the other intersection point in $\partial \tilde{S} \cap \gamma$ which is close with $y$. Let $\alpha$ be the component of the preimage of a seam in $U$ that intersects with $\beta$, and is the closest such arc from $\bar{x}$. Then the picture near $x$ is as shown in Figure 9. Give $\beta$ an orientation which points to the left in Figure 9, and note that the orientation of $t$ also points to the left.

Let $m$ be the middle point of $\overline{i(\bar{x})i(\bar{y})}$ in $\mathbb{H}^3$. Since $d_0(\bar{x},\bar{y})\geq\frac{R}{2}$ and $i|_{\widetilde{S}\cap \widetilde{W}}:\widetilde{S}\cap \widetilde{W} \rightarrow \mathbb{H}^3$ is an $(1+K\frac{\epsilon}{R},1)$-quasi-isometric embedding, $d(i(\bar{x}),i(\bar{y}))>\frac{R}{3}$ holds. So $d(i(\bar{x}),m),d(i(\bar{y}),m)>\frac{R}{6}$.

\begin{center}
\psfrag{a}[]{$x$} \psfrag{b}[]{$\bar{x}$} \psfrag{c}[]{$\gamma$}
\psfrag{d}[]{$\beta$} \psfrag{e}[]{$\alpha$} \psfrag{f}[]{$z$}
\psfrag{g}[]{$w$} \psfrag{h}[]{$t'$} \psfrag{i}[]{$t$} \psfrag{j}[]{$u$}
\includegraphics[width=3.5in]{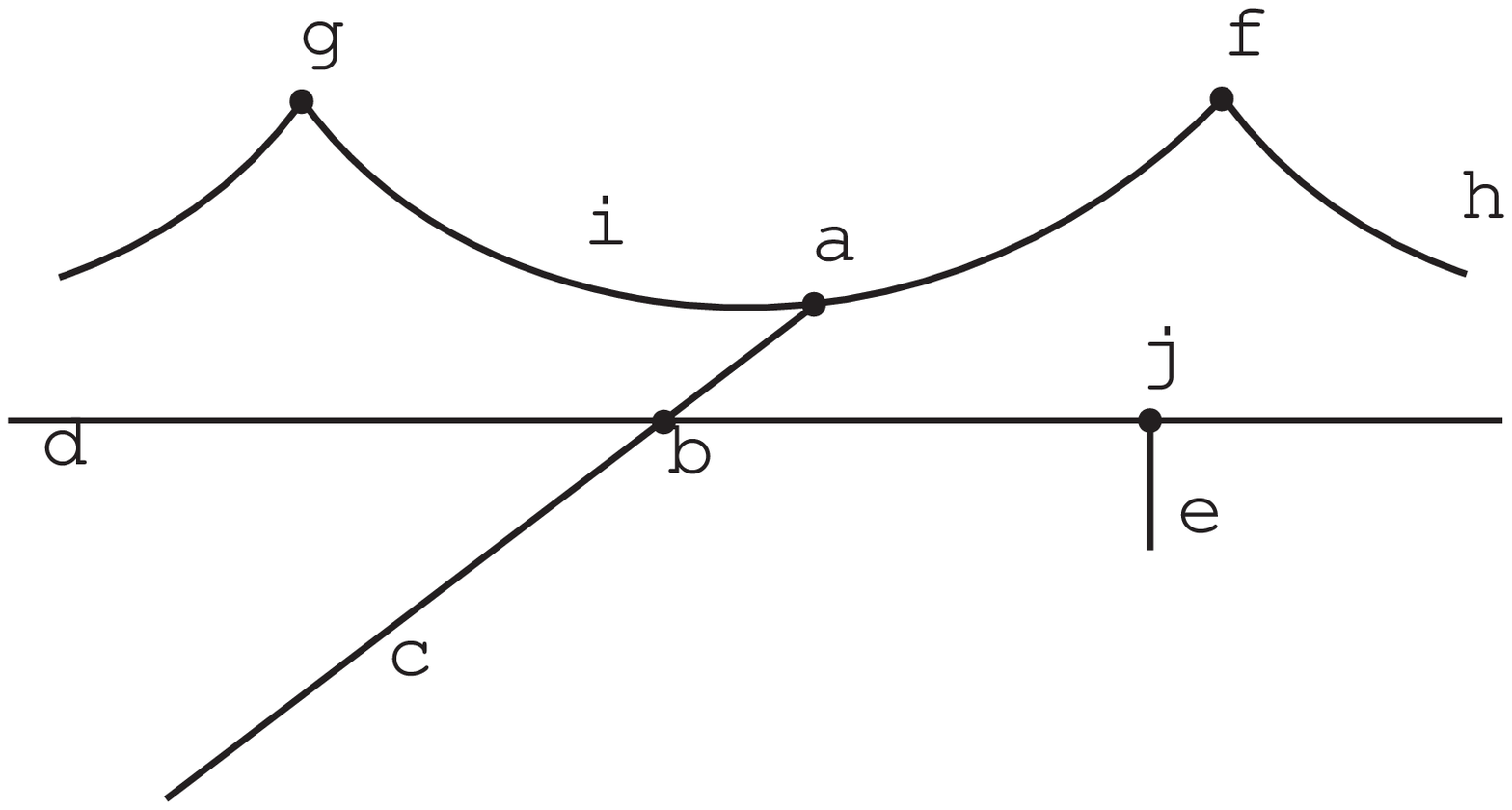}
 \centerline{Figure 9}
\end{center}

Now we make the following claim for $x$ and $\bar{x}$ , and the same estimations for $y$ and $\bar{y}$ also hold.
\begin{itemize}
\item $\frac{2}{3}d(m,i(x))<d_0(x,\bar{x})+\frac{1}{2}d_0(\bar{x},\bar{y})<\frac{3}{2}d(m,i(x))$.
\item $\angle{i(x)mi(\bar{x})}<\frac{\delta}{10}$.
\item $|\Theta(i_0(t),i_0(t'),\overline{i_0(x)i_0(y)})-\Theta(i(t),i(t'),\overline{i(x)m})|<\frac{\delta}{5}$.
\end{itemize}

This claim implies the statement of Proposition \ref{estimation} by the following argument. Since $\angle{i(x)mi(\bar{x})}<\frac{\delta}{10}$ and $\angle{i(y)mi(\bar{y})}<\frac{\delta}{10}$, $\angle{i(x)mi(y)}>\pi-\frac{\delta}{5}$ holds. By the first estimation in the claim,
\begin{equation}\label{15}
d(i(x),i(y))\leq d(i(x),m)+d(m,i(y))<\frac{3}{2}(d_0(x,\bar{x})+d_0(\bar{x},\bar{y})+d_0(y,\bar{y}))<2d_0(x,y),
\end{equation}
and
\begin{equation}\label{16}
d(i(x),i(y))\geq d(m,i(x))+d(m,i(y))-1 \geq \frac{2}{3}d_0(x,y)-1\geq \frac{1}{2}d_0(x,y).
\end{equation}

Moreover $\angle{i(x)mi(y)}>\pi-\frac{\delta}{5}$ implies $\angle{mi(x)i(y)}<\frac{\delta}{5}$. Then the third estimation in the claim implies $|\Theta(i_0(t),i_0(t'),\overline{i_0(x)i_0(y)})-\Theta(i(t),i(t'),\overline{i(x)i(y)})|<\delta$.

Now, we need only to prove the claim, and there are two subcases to consider: $d_0(x,\bar{x})<\frac{\delta}{1000}$ and $d_0(x,\bar{x})\geq\frac{\delta}{1000}$.

{\bf Subcase I.} $d_0(x,\bar{x})<\frac{\delta}{1000}$.

In this subcase, there might be a big difference between $\angle{i_0(w)i_0(x)i_0(\bar{x})}$ and $\angle{i(w)i(x)i(\bar{x})}$. However, since $d_0(x,\bar{x})<\frac{\delta}{1000}$ is very small, it will not affect the estimation very much.

By the first estimation of Proposition \ref{piece}, $d(m,i(\bar{x}))>\frac{R}{6}$ and $\frac{3}{4}d(m,i(\bar{x}))<\frac{1}{2}d_0(\bar{x},\bar{y})<\frac{4}{3}d(m,i(\bar{x}))$ hold. By the first and the second estimation of Proposition \ref{base1}, we also have $d(i(x),i(\bar{x}))<\frac{\delta}{250}$.

So
\begin{equation}\label{17}
\frac{2}{3}d(m,i(x))< \frac{2}{3}(d(m,i(\bar{x}))+d(i(\bar{x}),i(x)))< \frac{3}{4}d(m,i(\bar{x}))
 < \frac{1}{2}d_0(\bar{x},\bar{y})< d_0(x,\bar{x})+\frac{1}{2}d_0(\bar{x},\bar{y}),
\end{equation}
and
\begin{equation}\label{18}
d_0(x,\bar{x})+\frac{1}{2}d_0(\bar{x},\bar{y})<\frac{\delta}{1000}+\frac{4}{3}d(m,i(\bar{x}))
\leq \frac{\delta}{1000}+\frac{4}{3}(d(m,i(x))+d(i(x),i(\bar{x})))<\frac{3}{2}d(m,i(x))
\end{equation}
hold, thus the first estimation in the claim is true.

Moreover, since $d(i(x),i(\bar{x}))<\frac{\delta}{250}$ and $d(m,i(\bar{x}))>\frac{R}{6}$, $\angle{i(x)mi(\bar{x})}<\frac{\delta}{10}$ clearly holds.

Now it remains to show the third estimation in the claim.

Let $w''$ be the nearest point projection of $i(w)$ on $i(\beta)$. By the choice of the orientation of $t$, $d_0(w,x)\geq d_0(z,x)$ holds, so we have $d(i(w),i(x))\geq \frac{R}{4}$. Since $d(i(w),w'')\leq 2$, we have $\angle{i(w)i(x)w''}<10e^{-\frac{R}{4}}$. Moreover, since $d(w'',i(\bar{x}))\geq \frac{R}{4}-3$, $d(m,i(\bar{x}))\geq \frac{R}{6}$, and $d(i(x),i(\bar{x}))<\frac{\delta}{250}$, $\angle{i(x)w''i(\bar{x})},\angle{i(x)mi(\bar{x})}<\delta\cdot e^{-\frac{R}{6}}$ holds.

Let $\vec{v}_1$ be the tangent vector of $i(t)$ at $i(x)$, $\vec{v}_2$ be the tangent vector of $\overline{i(x)m}$ at $i(x)$, and $\vec{v}_3$ be the tangent vector of $\overline{i(x)w''}$ at $i(x)$. Let $\vec{u}_1$ be the tangent vector of $i(\beta)$ at $i(\bar{x})$ and $\vec{u}_2$ be the tangent vector of $\overline{i(\bar{x})m}$ at $i(\bar{x})$. For two points $p,q \in \mathbb{H}^3$ and $\vec{v}\in T_p\mathbb{H}^3$, we will use $\vec{v}@ q$ to denote the parallel transportation of $\vec{v}$ to $q$ along $\overline{pq}$.

Then
\begin{equation}\label{19}
\begin{split}
& |\angle{i(w)i(x)m}-\angle{w''i(\bar{x})m}|=|\Theta(\vec{v}_1,\vec{v}_2)-\Theta(\vec{u}_1,\vec{u}_2)|\\
 \leq &\Theta(\vec{v}_1,\vec{v}_3)+\Theta(\vec{v}_3,\vec{u}_1@ i(x))+\Theta(\vec{v}_2,\vec{u}_2@ i(x))\\
 \leq &\angle{i(w)i(x)w''}+\Theta(\vec{v}_3@ w'',\vec{u}_1@w'')+\Theta(\vec{u}_1@w'',\vec{u}_1@ i(x)@w'')\\
 & +\Theta(\vec{v}_2@m,\vec{u}_2@m)+\Theta(\vec{u}_2@m,\vec{u}_2@i(x)@m)\\
 \leq & \angle{i(w)i(x)w''}+\angle{i(x)w''i(\bar{x})}+d(i(x),i(\bar{x}))+\angle{i(x)mi(\bar{x})}+d(i(x),i(\bar{x}))\\
 \leq & 10e^{-\frac{R}{4}}+\delta\cdot e^{-\frac{R}{6}}+\frac{\delta}{250}+\delta\cdot e^{-\frac{R}{6}}+\frac{\delta}{250}< \frac{\delta}{100}.
\end{split}
\end{equation}
Here $\Theta(\vec{u}_1@w'',\vec{u}_1@i(x)@w'')<d(i(x),i(\bar{x}))$ holds by Proposition 4.1 of \cite{KM1}.

Let $w'$ be the nearest point projection of $i_0(w)$ on $i_0(\beta)$, then a similar (actually easier) argument implies
\begin{equation}\label{20}
|\angle{i_0(w)i_0(x)i_0(y)}-\angle{w'i_0(\bar{x})i_0(y)}|<\frac{\delta}{100}.
\end{equation}

Note that the first coordinate of $|\Theta(i_0(t),i_0(t'),\overline{i_0(x)i_0(y)})-\Theta(i(t),i(t'),\overline{i(x)m})|$ equals $|\angle{i_0(w)i_0(x)i_0(y)}-\angle{i(w)i(x)m}|$, while the second estimation of Proposition \ref{piece} implies $|\angle{w'i_0(\bar{x})i_0(\bar{y})}-\angle{w''i(\bar{x})i(\bar{y})}|<\frac{\delta}{100}$. So
\begin{equation}\label{21}
\begin{split}
& |\angle{i_0(w)i_0(x)i_0(y)}-\angle{i(w)i(x)m}|\\
\leq & |\angle{i_0(w)i_0(x)i_0(y)}-\angle{w'i_0(\bar{x})i_0(\bar{y})}|+|\angle{i(w)i(x)m}-\angle{w''i(\bar{x})m}|+|\angle{w'i_0(\bar{x})i_0(\bar{y})}-\angle{w''i(\bar{x})m}|\\
\leq &\frac{\delta}{100}+\frac{\delta}{100}+\frac{\delta}{100}<\frac{\delta}{10}.
\end{split}
\end{equation}

Let $m'$ be the nearest point projection of $m$ on $P_{i(t)i(t')}$, and $\bar{x}'$ be the nearest point projection of $i(\bar{x})$ on $P_{i(t)i(t')}$, then $d(\bar{x}',i(\bar{x}))\leq \frac{\delta}{250}$. Note that the second coordinate of $|\Theta(i_0(t),i_0(t'),\overline{i_0(x)i_0(y)})-\Theta(i(t),i(t'),\overline{i(x)m})|$ equals $\angle{mi(x)m'}$.

Let $\vec{n}$ be the normal vector of $P_{i(t)i(t')}$ at $i(x)$ and $\vec{n}'$ be the normal vector of $P_{i(x)i(\beta)}$ (hyperbolic plane containing $i(x)$ and $i(\beta)$) at $i(\bar{x})$ (with almost coincide orientation with $\vec{n}$). Then by the estimation of $\Theta(P_{i(t)i(t')},P_{i(x)i(\beta)})$ in equation \eqref{14}, we have
\begin{equation}\label{22}
\Theta(\vec{n}'@i(x),\vec{n})=\Theta(P_{i(t)i(t')},P_{i(x)i(\beta)})\leq 50\frac{\epsilon}{R}.
\end{equation}
So
\begin{equation}\label{23}
\angle{(\vec{n}'@\bar{x}',\vec{n}@\bar{x}')}\leq \angle{(\vec{n}'@\bar{x}',\vec{n}'@i(x)@\bar{x}')}+\angle{(\vec{n}'@i(x),\vec{n})}\leq \frac{\delta}{250}+50\frac{\epsilon}{R}\leq \frac{\delta}{200}.
\end{equation}

Then we have
\begin{equation}\label{24}
\begin{split}
\sinh{d(m,m')}& \leq \sinh{d(i(\bar{x}),\bar{x}')}\cosh{d(i(\bar{x}),m)}+\cosh{d(i(\bar{x}),\bar{x}')}\sinh{d(i(\bar{x}),m)}\sin{\angle{(\vec{n}'@\bar{x}',\vec{n}@\bar{x}')}}\\
& \leq \frac{\delta}{100}\cosh{d(i(\bar{x}),m)}+\frac{\delta}{100}\sinh{d(i(\bar{x}),m)}.
\end{split}
\end{equation}

This inequality implies
\begin{equation}\label{25}
\sin{\angle{mi(x)m'}} =\frac{\sinh{d(m,m')}}{\sinh{d(i(x),m)}} \leq \frac{\frac{\delta}{100}\cosh{d(i(\bar{x}),m)}+\frac{\delta}{100}\sinh{d(i(\bar{x}),m)}}{\sinh{(d(i(\bar{x}),m)-\frac{\delta}{250})}} \leq \frac{\delta}{20}.
\end{equation}
So $\angle{mi(x)m'}\leq \frac{\delta}{10}$, and we finish the proof in the first subcase.

{\bf Subcase II.} $d_0(x,\bar{x})\geq\frac{\delta}{1000}$.

By the first two estimations of Lemma \ref{base1} and $d_0(x,\bar{x})\geq\frac{\delta}{1000}$, we have that
\begin{equation}\label{26}
\frac{3}{4}d(i(x),i(\bar{x}))<d_0(x,\bar{x})<\frac{4}{3}d(i(x),i(\bar{x}))
\end{equation}
holds, while the first estimation of Theorem \ref{piece} implies
\begin{equation}\label{27}
\frac{3}{4}d(m,i(\bar{x}))<\frac{1}{2}d_0(x,y)<\frac{4}{3}d(m,i(\bar{x})).
\end{equation}

Now let us estimate $\angle{mi(\bar{x})i(x)}$. Let $x'$ be the nearest point projection of $i_0(x)$ on $i_0(\beta)$ and $x''$ be the nearest point projection of $i(x)$ on $i(\beta)$.

If $d(x',i_0(\bar{x}))\geq R$, then $d(i_0(x),i_0(\bar{x}))\geq R$, and $d(i(x),i(\bar{x}))\geq \frac{R}{2}$ hold. Since $d(i_0(x),x')$, $d_0(i(x),x'')<2$, we have $\angle{i_0(x)i_0(\bar{x})x'}, \angle{i(x)i(\bar{x})x''}\leq 10e^{-\frac{R}{2}}$. By the second estimation of Proposition \ref{piece}, $\angle{i_0(x)i_0(\bar{x})x'}\leq 10e^{-\frac{R}{2}}$ implies $\angle{mi(\bar{x})x''}\geq \pi-10e^{-\frac{R}{2}}-(\frac{\delta}{300})^2$. So $\angle{mi(\bar{x})i(x)}\geq \pi-20e^{-\frac{R}{2}}-(\frac{\delta}{300})^2 \geq \pi-\frac{\delta}{25}$.

If $d(x',i_0(\bar{x}))<R$, then $x'$ and $i_0(\bar{x})$ lie in the image of at most two adjacent $1$-cells in $\widetilde{Z}$, under the nearest point projection to $i_0(\beta)$. Then
the first two estimations in Lemma \ref{base1} implies
\begin{equation}\label{28}
|\Theta(i_0(\beta),i_0(\alpha),\overline{i_0(\bar{x})i_0(x)})-\Theta(i(\beta),i(\alpha),\overline{i(\bar{x})i(x)})|<10^4(\frac{\epsilon}{R})^{\frac{1}{4}}\delta^{-\frac{1}{2}}.
\end{equation}
The second estimation of Theorem \ref{piece} gives
\begin{equation}\label{29}
|\Theta(i_0(\beta),i_0(\alpha),\overline{i_0(\bar{x})i_0(\bar{y})})-\Theta(i(\beta),i(\alpha),\overline{i(\bar{x})m})|<(\frac{\delta}{300})^2.
\end{equation}
Then an elementary computation in spherical geometry gives
\begin{equation}\label{30}
\angle{mi(\bar{x})i(x)}\geq \pi-3\sqrt{(\frac{\delta}{300})^2}-3\sqrt{10^4(\frac{\epsilon}{R})^{\frac{1}{4}}\delta^{-\frac{1}{2}}}\geq \pi-\frac{\delta}{25}.
\end{equation}
So $\angle{mi(\bar{x})i(x)}\geq \pi-\frac{\delta}{25}$ always holds.

By \eqref{26}, \eqref{27} and \eqref{30}, we have:
\begin{equation}\label{31}
\begin{split}
& \frac{2}{3}d(m,i(x))<\frac{2}{3}(d(m,i(\bar{x})+d(i(\bar{x}),i(x)))<\frac{1}{2}d_0(x,y)+d_0(x,\bar{x})\\
& <\frac{4}{3}(d(m,i(\bar{x}))+d(i(x),i(\bar{x})))<\frac{4}{3}(d(m,i(x))+1)<\frac{3}{2}d(m,i(x)).
\end{split}
\end{equation}

Moreover, $\angle{mi(\bar{x})i(x)}\geq \pi-\frac{\delta}{25}$ implies $\angle{i(x)mi(\bar{x})}<\frac{\delta}{25}<\frac{\delta}{10}$ and $\angle{mi(x)i(\bar{x})}<\frac{\delta}{25}$.

The estimations in Lemma \ref{base1} also imply the following estimation (by considering the cases that $d_0(x',i_0(\bar{x}))\geq R$ and $d_0(x',i_0(\bar{x}))< R$)
\begin{equation}\label{32}
|\Theta(i_0(t),i_0(t'),\overline{i_0(x)i_0(\bar{x})})-\Theta(i(t),i(t'),\overline{i(x)i(\bar{x})})|<10^4(\frac{\epsilon}{R})^{\frac{1}{4}}\delta^{-1}<\frac{\delta}{25}.
\end{equation}
Then equation \eqref{32} and $\angle{mi(x)i(\bar{x})}<\frac{\delta}{25}$ together give the desired estimation
\begin{equation}\label{33}
|\Theta(i_0(t),i_0(t'),\overline{i_0(x)i_0(y)})-\Theta(i(t),i(t'),\overline{i(x)m})|<\frac{\delta}{5}.
\end{equation}
\end{proof}

Now we are ready to prove Theorem \ref{technical}, which finishes the proof of our main theorem (Theorem \ref{main}).

\begin{proof}
Choose constants $\hat{\epsilon}>0$ and $\hat{R}>0$ such that Proposition \ref{estimation} holds for $\delta=(\frac{\pi}{360})^2$.

For any two points $x,y\in \widetilde{Z}^{(0)}$, let $\gamma$ be the shortest path in $(\widetilde{Z},d_0)$ from $x$ to $y$, $\gamma'$ be the modified path of $\gamma$, and $y_1,y_2,\cdots,y_m$ be the corresponding modified sequence.

Note that since $x,y\in \widetilde{Z}^{(0)}$, $d_0(x,y_1),d_0(y_m,y)\geq \frac{R}{2}$ holds. Then $\gamma'$ is a concatenation of geodesic arcs $\gamma_0,\gamma_1,\cdots,\gamma_m$ in $(\widetilde{Z},d_0)$, which connect the sequence of points $y_0=x,y_1,y_2,\cdots,y_m,y_{m+1}=y$. Here each $\gamma_i$ lies in a piece of $\widetilde{Z}$, with length greater than $\frac{R}{128}$, for $i=0,1,\cdots,m$. Moreover, in the proof of Proposition \ref{quasi1}, we have shown that, for adjacent geodesic arcs $i_0(\gamma_i)$ and $i_0(\gamma_{i+1})$, the angle between them is greater than $\frac{\pi}{18}$.

In the first estimation of Proposition \ref{estimation}, we showed that $d(i(y_i),i(y_{i+1}))<2d_0(y_i,y_{i+1})$. While equation \eqref{8} in the proof of Proposition \ref{quasi1} implies that $\sum_{i=0}^m d(i_0(y_i),i_0(y_{i+1}))\leq \frac{10}{9}d(i_0(x),i_0(y))\leq \frac{10}{9}d_0(x,y)$. So \begin{equation}\label{34}
d(i(x),i(y))\leq \sum_{i=0}^m d(i(y_i),i(y_{i+1}))\leq 2\sum_{i=0}^m d_0(y_i,y_{i+1})\leq 3d_0(x,y).
\end{equation}

On the other hand, the second estimation in Proposition \ref{estimation} implies the following statement. Let $\vec{u}_1\in S^2$ that corresponds with the tangent vector of $i_0(\gamma_i)$ at $i_0(y_{i+1})$ (under frame $(\vec{e}_1,\vec{e}_2,\vec{e}_3)$ at $i_0(y_{i+1})$), and $\vec{u}_2\in S^2$ that corresponds with the tangent vector of $\overline{i(y_{i+1})i(y_i)}$ at $i(y_{i+1})$ (under frame $(\vec{e}_1',\vec{e}_2',\vec{e}_3')$ at $i(y_{i+1})$), then $\Theta(\vec{u}_1,\vec{u}_2)\leq \frac{\pi}{90}$. Since we have shown that the angle between $i_0(\gamma_i)$ and $i_0(\gamma_{i+1})$ is greater than $\frac{\pi}{18}$ in the Proof of Proposition \ref{quasi1}, the angle between $\overline{i(y_{i+1})i(y_i)}$ and $\overline{i(y_{i+1})i(y_{i+2})}$ is greater than $\frac{\pi}{36}$.

Since the length of $\overline{i(y_{i+1})i(y_i)}$ equals $d(i(y_i),i(y_{i+1}))$, which is greater than $\frac{R}{256}$, then Lemma 4.8 of \cite{LM} implies
\begin{equation}\label{35}
\begin{split}
& d(i(x),i(y))\geq \sum_{i=0}^m d(i(y_i),i(y_{i+1}))-2m(\log{(\csc{\frac{\pi}{72}})}+1)\\
& \geq \frac{1}{2}\sum_{i=0}^m d(i(y_i),i(y_{i+1}))\geq \frac{1}{4} \sum_{i=0}^m d_0(y_i,y_{i+1})\geq \frac{1}{4}d_0(x,y).
\end{split}
\end{equation}

So $i:(\widetilde{W},d_0|_{\widetilde{W}})\rightarrow (\mathbb{H}^3, d_{\mathbb{H}^3})$ is a quasi-isometric embedding.
\end{proof}

\bibliographystyle{amsalpha}

\end{document}